\documentclass[12pt]{amsart}

\usepackage[margin=1in]{geometry}  
\usepackage{graphicx}              
\usepackage{amsmath}               
\usepackage{amsfonts}              
\usepackage{amsthm}                
\usepackage{amssymb}
\usepackage{amscd}
\usepackage{listings}
\usepackage{color} 
\usepackage{mathrsfs}  
\usepackage{hyperref}
\hypersetup{
    colorlinks=true, 
    linktoc=all,     
    linkcolor=black,  
    citecolor=blue,
    urlcolor=blue
}
\usepackage{tikz}
\usetikzlibrary{arrows}

\newtheorem{thm}{Theorem}[section]
\newtheorem{lem}[thm]{Lemma}
\newtheorem{prop}[thm]{Proposition}
\newtheorem{cor}[thm]{Corollary}

\newtheorem{defn}[thm]{Definition}

\newcommand\QQ{{\mathbb Q}}
\newcommand{\RR}{\mathbb{R}}      
\newcommand{\ZZ}{\mathbb{Z}}      
\newcommand\CC{{\mathbb C}}
\newcommand\NN{{\mathbb N}}

\newcommand\gr{{\mathrm{Gr}_G}}

\newcommand{\norm}[1]{\left\lVert#1\right\rVert}

\makeatletter
\newcommand*{\defeq}{\mathrel{\rlap{%
                     \raisebox{0.3ex}{$\m@th\cdot$}}%
                     \raisebox{-0.3ex}{$\m@th\cdot$}}%
                     =}

\newcommand{\g}{\mathfrak{g}}

\newcommand{\spec}{\mathrm{Spec}\ }

\begin{document}

\nocite{*}

\title{Crystal Combinatorics and Geometric Satake}

\author{Eric Chen}
\address[Eric Chen]{University of California, Berkeley}
\email{a5584266@berkeley.edu}

\begin{abstract} This is an REU paper written for the University of Chicago REU, summer 2017. The main purpose of this note is to collect some of the many of combinatorial models for MV cycles that exist in the literature. In particular, we will investigate MV polytopes, preprojective algebras, crystals, and LS galleries. We also compute explicitly some examples to help elucidate the theory and the connections between these models. 

\end{abstract}

\maketitle

\tableofcontents

\section{Introduction}
Let $G$ be a complex reductive algebraic group. The geometric Satake correspondence states that the category of certain equivariant perverse sheaves on the affine Grassmannian $\mathrm{Gr}_G$ is naturally equivalent to the category of representations of $\hat{G}$, the Langlands dual group of $G$ \cite{Zhu}. Among the many miraculous connections provided by geometric Satake, Mirkovi\'c and Vilonen have constructed subvarieties of $\mathrm{Gr}_G$ (now known as MV cycles) that serve as a geometric avatar for the theory of weight spaces of $\hat{G}$ \cite{Mirkovic-Vilonen}. 

Although MV cycles are easy to define, they are often difficult to explicitly calculate. Anderson, Baumann, and Kamnitzer's works remedy this by defining combinatorial models which they call MV polytopes, and Kamnitzer's recent work further relates these polytopes to the representation theory of quivers \cite{Baumann-Kamnitzer}. It is also revealed in these papers that the MV polytopes carry natural crystal structures, and as crystals they are isomorphic to certain crystals defined by Kashiwara in \cite{Kashiwara}.

On the other hand, Gaussent and Littelmann have independently developed another combinatorial model for MV cycles known as LS galleries \cite{Gaussent-Littelmann}. However, the relationship between LS galleries and MV polytopes is still largely unclear; a step towards this direction is taken in \cite{Baumann-Gaussent}, in which a crystal structure is defined on galleries, and compared with the crystal structure of MV cycles. 

The goal of the present paper is two-fold: first, give a self-contained, brief, and accurate formulation of the geometric Satake equivalence. This motivates our main object of study: the MV cycles. We introduce combinatorial models such as MV polytopes, quiver representations, and galleries mentioned above, and define their crystal combinatorics. Secondly, we wish to compute many examples: it is obvious in the present literature that there is a lack of examples worked out in detail. We hope that not only the combinatorics, but geometric Satake itself can be elucidated via these examples.

\section{Geometric Satake and MV Cycles} In this section, we review the basic setup for the geometric Satake correspondence in order to motivate MV cycles. Our treatment here essentially follows the exposition of Zhu in \cite{Zhu}, complemented with some explicit examples. We will only point out the essential properties of the affine Grassmannian needed to introduce MV cycles, as a full discussion would take us far afield. However, for readers who are interested in the affine Grassmannian itself, we provide references to more complete sources. 

\subsection{The Affine Grassmannian} \label{subsect: affineGrass} Let $k$ be a field, and let $G$ be a smooth affine $k$-group scheme. The affine Grassmannian $\gr$ is a central object of study in geometric representation theory, and admits multiple descriptions via moduli problems of lattices, $G$-bundles, and loop groups. We will use in this note the following geometric interpretation of $\gr$:

\begin{defn} \label{def: Grassmannian} The affine Grassmannian $\gr$ is the presheaf defined on a $k$-algebra $R$ by
$$\gr(R) = \left\{(\mathscr{E}, \beta): \begin{matrix}\mathscr{E} \text{  is a } G\text{-torsor on } D_R\text{, and }\\ 
 \beta: \mathscr{E}|_{D^*_R} \stackrel{\sim}{\to} \mathscr{E}^0|_{D^*_R} \text{  is a trivialization.}\end{matrix}\right\}$$
here $D_R = \spec R[[t]]$ and $D^*_R = \spec R((t))$ denote the formal disc and the formal punctured disc, respectively, and $\mathscr{E}^0$ is the trivial $G$-torsor on $D_R$. 
\end{defn}

Geometrically, $\gr$ is represented by an infinite dimensional \textit{ind-scheme}, i.e., the colimit of a sequence of closed immersions of schemes. However, we will only be interested in certain finite dimensional subvarieties, which are in fact projective varieties if we assume $G$ is reductive (see Theorem 1.2.2 in \cite{Zhu}). Our goal in this section is to describe and motivate these finite dimensional subvarieties of $\gr$. 

Although the geometric description of the affine Grassmannian is conceptually straightforward, we will need an algebraic description as well to perform computations. To this end, we first observe that every $G$-torsor on $D_R$ can be trivialized over $D_{R'}$ for some \'etale cover $\spec R' \to \spec R$, so it is natural to consider the presheaf $LG$ which classifies triples
$$LG(R) := \left\{(\mathscr{E}, \beta, \epsilon): (\mathscr{E}, \beta) \in \gr(R), \epsilon: \mathscr{E} \stackrel{\sim}{\to} \mathscr{E}^0\right\}.$$
Identifying $G(R[[t]])$ with $\mathrm{Aut}_{D_R}(\mathscr{E}^0)$ and $G(R((t)))$ with $\mathrm{Aut}_{D^*_R}(\mathscr{E}^0|_{D^*_R})$, it's not hard to see that $LG$ is isomorphic to the \textit{loop group presheaf}
$$R \mapsto G(R((t))),$$
and the redundancy of the variable $\epsilon$ can be canceled by a natural action of the \textit{positive loop group}
$$L^+G: R \mapsto G(R[[t]])$$
This intuitive discussion can be made rigorous by Proposition 1.3.6 in \cite{Zhu}:
\begin{prop} The affine Grassmannian $\gr$ is canonically identified with the fpqc quotient $[LG/L^+G]$. 
\end{prop}
This makes computations more feasible in the following sense: if $k$ is separably closed, then we can (and do) identify the $k$-points of $\gr$ with the quotient space $G(k((t)))/G(k[[t]])$. 
\\

For the rest of the section, let $k$ be an algebraically closed field, and fix the notation $\mathcal{O} = k[[t]], F = k((t))$, $G$ a constant reductive $k$-group scheme, and $\mathrm{Gr} = \gr$. We also fix an embedding of a maximal torus inside a Borel $T \subset B \subset G$, and associate to $T$ its weight lattice $\mathbf{X}^*(T)$ and coweight lattice $\mathbf{X}_*(T)$. Recall that via Cartan decomposition we have an identification
$$G(\mathcal{O})\backslash G(F)/G(\mathcal{O}) \cong \mathbf{X}_*(T)/W \cong \mathbf{X}_*(T)^+$$
of the double coset space with the positive coweights of $T$, where $W$ is the Weyl group of $G$. We think of the double coset space on the left hand side as the left coset space for the action of $G(\mathcal{O})$ on $\mathrm{Gr}$. In terms of Definition \ref{def: Grassmannian}, this decomposition can be realized as follows. Let $\mathscr{E}_1, \mathscr{E}_2$ be two $G$-torsors over $D = D_R$, and 
$$\beta: \mathscr{E}_1|_{D^*} \stackrel{\cong}{\longrightarrow} \mathscr{E}_2|_{D^*}$$
an isomorphism over the punctured disc. If $(\mathscr{E}_1, \beta_1), (\mathscr{E}_2, \beta_2)$ are two sections of $\gr(R)$, then $\beta_2\beta\beta_1^{-1}$ defines an automorphism of the trivial $G$-torsor $\mathscr{E}^0|_{D^*}$, which is precisely an element of $G(F)$. Different choices of $\beta_1$ and $\beta_2$ correspond to left and right multiplication by elements of $G(\mathcal{O})$, so $\beta$ defines an element
$$\mathrm{Inv}(\beta) \in G(\mathcal{O})\backslash G(F)/G(\mathcal{O}) \cong \mathbf{X}_*(T)^+$$
which we call the \textit{relative position} of $\beta$. 

Let $\mu$ be a coweight. We define the \textit{(spherical) Schubert variety} $\mathrm{Gr}_{\leq \mu}$ as the subset
$$\mathrm{Gr}_{\leq \mu} := \left\{(\mathscr{E}, \beta) \in \mathrm{Gr}: \mathrm{Inv}(\beta) \leq \mu\right\}.$$
By Proposition 2.1.4 of \cite{Zhu}, we know that $\mathrm{Gr}_{\leq \mu}$ is in fact a closed subset, so we may endow it with the reduced subscheme structure. In fact, $\mathrm{Gr}_{\leq \mu}$ is a projective variety, and can be understood as the Zariski closure of the \textit{Schubert cell}
$$\mathrm{Gr}_\mu := \left\{(\mathscr{E}, \beta) \in \mathrm{Gr}: \mathrm{Inv}(\beta) = \mu\right\}.$$

We can also describe the Schubert cells algebraically. If $\mu$ is a coweight, then it defines an element $t^\mu := \mu(t)$ in $G(F)$, which represents the double coset
$$G_\mu = G(\mathcal{O}) \cdot t^\mu \cdot G(\mathcal{O}),$$
and we in fact have that $G_\mu = \mathrm{Gr}_\mu$ (see Proposition 2.1.5 in \cite{Zhu}). For the sake of concreteness, we quickly prove that these Schubert cells are in fact finite dimensional (quasi-projective) varieties:
\begin{prop} Let $2\rho$ be the sum of all positive roots. Then $\mathrm{Gr}_\mu$ has dimension $(2\rho, \mu)$.
\end{prop}
\begin{proof} Note that the stabilizer of $t^\mu$ in $L^+G$ is precisely 
$$L^+G \cap t^\mu \cdot L^+G \cdot t^{-\mu}$$
Thus, we can identify the tangent space of $\mathrm{Gr}_\mu$ at $t^\mu$ with
$$\g(\mathcal{O})/\left(\g(\mathcal{O}) \cap \mathrm{Ad}_{t^\mu}\g(\mathcal{O})\right) = \bigoplus_{(\alpha, \mu) \geq 0} \g_\alpha(\mathcal{O})/t^{(\alpha, \mu)}\g_\alpha(\mathcal{O})$$
where the direct sum is taken over positive roots $\alpha$ of $G$ and $\g_\alpha$ denotes the corresponding root space. The dimension of the vector space on the right hand side is precisely $(2\rho, \mu)$, which is what we wanted to show.
\end{proof}

We will also need to consider certain infinite dimensional orbits. Let $U$ be the unipotent radical of $B$, then for every coweight $\lambda$, we can consider double cosets of the form 
$$S_\lambda := LU \cdot t^\lambda \cdot G(\mathcal{O}) = U(F) \cdot t^\lambda \cdot G(\mathcal{O})$$
where $LU$ acts on $t^\lambda$ via restriction from $LG$. 
\\

\noindent
\textbf{Example.} Let $G = \mathrm{PGL}_3$, and $\mu = (n_1,n_2,n_3)$ be a dominant coweight; this is the coweight
$$\mu: t \mapsto t^\mu = \begin{pmatrix}t^{n_1} & 0 & 0\\ 0 & t^{n_2} & 0\\ 0& 0&t^{n_3} \end{pmatrix}$$
for some $n_1 \geq n_2 \geq n_3$. Then the Schubert variety indexed by $\mu$ is precisely
$$\mathrm{Gr}_{\leq \mu} = \bigcup_{\lambda \leq \mu} \mathrm{Gr}_\lambda = \left\{g \in G(F)/G(\mathcal{O}): \begin{matrix} \nu(g_{ij}) \geq n_3,\\ \nu(2 \times 2\text{-minors}) \geq n_2+n_3, \\\nu(\mathrm{det}(g_{ij})) = n_1+n_2+n_3\end{matrix} \right\}$$
where $\nu$ is the standard valuation on $k((t))$. Now let's consider the infinite dimensional orbit $S_\lambda$, where $\lambda = (m_1, m_2, m_3)$:
$$S_\lambda = LU \cdot t^\lambda \cdot G(\mathcal{O}) = \begin{pmatrix}t^{m_1} & x & y\\ 0 & t^{m_2} &z\\0&0&t^{m_3} \end{pmatrix} \cdot G(\mathcal{O})$$
for some $x, y, z \in F = k((t))$. In the following sections, we will be interested in computing various intersections $S_\lambda \cap \mathrm{Gr}_{\leq \mu}$, which is feasible given the explicit presentation here. 
\\

We conclude this section by defining the convolution structure on $\gr$. Define the \textit{convolution Grassmannian} $\gr \tilde{\times} \gr$ as the presheaf whose $R$-points are
$$\left\{(\mathscr{E}_1, \beta_1, \mathscr{E}_2, \beta_2): \begin{matrix}(\mathscr{E}_1, \beta_1) \in \gr(R), \mathscr{E}_2 \text{ a } G\text{-torsor, and }\\ \beta_2: \mathscr{E}_2|_{D_R^*} \stackrel{\sim}{\to} \mathscr{E}_1|_{D_R^*}\end{matrix}\right\}$$
Then we have a map $m: \gr \tilde{\times} \gr \to \gr$ defined on $R$-points by
$$(\mathscr{E}_1, \beta_1, \mathscr{E}_2, \beta_2) \mapsto (\mathscr{E}, \beta_1\beta_2)$$
which we call the \textit{convolution map}. Similarly, we can define an $n$-fold convolution Grassmannian 
$$(\gr \tilde{\times} \cdots \tilde{\times} \gr)(R) := \left\{(\mathscr{E}_i, \beta_i)_{i=1}^n: \beta_i: \mathscr{E}_i|_{D_R^*} \stackrel{\sim}{\to} \mathscr{E}_{i-1}|_{D_R^*}\right\}$$
where $\mathscr{E}_0 = \mathscr{E}^0$ is the trivial $G$-torsor on $D_R$. We then get an $n$-fold convolution map
\begin{align*}
m: \gr \tilde{\times} \cdots \tilde{\times} \gr \to \gr\\
(\mathscr{E}_i, \beta_i)_{i=1}^n \mapsto (\mathscr{E}_n, \beta_1\cdots\beta_n)
\end{align*}

\noindent
\textbf{Remark.} The definition of the affine Grassmannian given in Definition \ref{def: Grassmannian} begs the question of whether we can find a moduli description of $\gr$ that is global, i.e., relative to some curve $X$. By an (infinitesimal) descent theorem of Beauville-Laszlo \cite{Beauville-Laszlo}, we can indeed reformulate $\gr$ in terms of the moduli of $G$-torsors on any curve $X$, along with a choice of trivialization away from a closed point $x$ in $X$; this is known as the \textit{Beilinson-Drinfeld Grassmannian} $\mathrm{Gr}_{G,x}$. If we choose $n$ distinct closed points $x_1, \ldots, x_n$ on some fixed curve $X$, then the local nature of Definition \ref{def: Grassmannian} will allow us to \textit{factorize} the Beilinson-Drinfeld Grassmannian as follows:
$$\mathrm{Gr}_{G, (x_1, \ldots, x_n)} \stackrel{\sim}{\longrightarrow} \gr \times \cdots \times \gr$$
where the left hand side classifies the following datum:
$$\left\{(\mathscr{E}_i, \beta_i)_{i=1}^n: \mathscr{E}_i \text{ are } G\text{-torsors on } X, \beta_i: \mathscr{E}_i|_{X \setminus x_i} \to \mathscr{E}^0|_{X \setminus x_i}\right\}$$
These properties characterize the Beilinson-Drinfeld Grassmannian as a \textit{factorization space}. The interested reader can find a complete treatment of the Beilinson-Drinfeld Grassmannian in Chapter 3 of \cite{Zhu}. 

\subsection{Geometric Satake} The geometric Satake equivalence as established by Mirkovi\'c and Vilonen states that the category of $G(\mathcal{O})$-equivariant perverse sheaves on the affine Grassmannian $\gr$ is equivalent to the category of representations of $\hat{G}$ as tensor categories. Our goal for this section is to decipher the statement of this result, and introduce the Mirkovi\'c-Vilonen cycles as a central character in this equivalence of categories. Since our goal is not to formally explain the general context for perverse sheaves, i.e., t-structures on triangulated categories, we will follow a ``shortest path" that leads to a precise statement of geometric Satake. As a consequence, certain formalisms may seem awkward as they are not discussed at an appropriate level of generality. Another simplification we will adopt is to consider only complex valued sheaves instead of $\overline{\QQ_\ell}$-sheaves; this will greatly streamline our exposition, at the cost of restricting to complex reductive groups. However, the ideas discussed here will easily generalize to the case of $\ell$-adic sheaves, and we refer the reader to the notes of \cite{Cesnavicius} for a comprehensive treatment of the general theory of $\ell$-adic perverse sheaves. 

As before, we fix $\mathcal{O} = \CC[[t]]$, $F = \CC((t))$, and $G$ a connected complex reductive group. We first introduce some basic definitions. 

Let $X$ be a scheme over $\CC$. Recall that a stratification of $X$ is a finite collection of locally closed subschemes $\{X_\alpha\}_{\alpha \in A}$ such that each $X_\alpha$ is smooth, $X = \bigcup_\alpha X_\alpha$, and $\overline{X_\alpha} = \bigcup_{\beta \in B} X_\beta$ for some subset $B \subseteq A$. 

\begin{defn} Let $\mathscr{F}$ be a sheaf on $X$. Then we say $\mathscr{F}$ is constructible if there exists a stratification $X = \bigcup_\alpha X_\alpha$ such that $\mathscr{F}|_{X_\alpha}$ is locally constant for each $X_\alpha$. 
\end{defn}

Let's consider the bounded derived category of constructible ($\CC$-valued) sheaves $\mathcal{D}^b_c(X, \CC)$. The objects in this category are complexes of sheaves $\mathscr{F}^\bullet$ such that each $\mathscr{H}^i(\mathscr{F}^\bullet)$ is constructible, and $\mathscr{F}^i = 0$ for $|i| >> 0$; morphisms in this category are chain maps, localized with respect to chain homotopies. From now on, by a \textit{sheaf} we will mean an object of $\mathcal{D}_c^b(X, \CC)$, and all pushforwards and pullbacks will be \textit{derived}. 

\begin{defn} A sheaf $\mathscr{F}^\bullet$ in $\mathcal{D}_c^b(X, \CC)$ is perverse if the following two conditions hold:
$$\mathrm{dim} \, \mathrm{supp} \, \mathscr{H}^{-i}(\mathscr{F}^\bullet) \leq i \ \text{ for all } i \in \ZZ$$
$$\mathrm{dim} \, \mathrm{supp} \, \mathscr{H}_c^{i}(\mathscr{F}^\bullet) \leq i \ \text{ for all } i \in \ZZ$$
We denote the full subcategory of perverse sheaves on $X$ by $\mathrm{Perv}(X)$. 
\end{defn}

This definition may seem artificial without context, but we give two reasons, one conceptual and the other technical, that outline the importance of considering the subcategory $\mathrm{Perv}(X) \subset \mathcal{D}^b_c(X, \CC)$. If $X$ is a smooth connected complex variety, the Riemann-Hilbert correspondence states that there is a fully faithful functor $dR: \mathrm{Conn}_{flat}(X) \to \mathcal{D}^b_c(X, \CC)$ sending a bundle with flat connection to a complex of sheaves. In particular, $dR$ gives an equivalence of categories between the subcategory of flat connections with regular singularities and $\mathrm{Perv}(X)$. Since one identifies flat connections with representations of the fundamental group (i.e., local systems) via monodromy, one can think of perverse sheaves on a general scheme as local systems. 

The notion of perversity is also important from the perspective of geometric Satake itself. We will see later that we want to compare a subcategory of $\mathrm{Perv}(X)$ with $\mathrm{Rep}_{\hat{G}}$, the category of finite dimensional complex representations of $\hat{G}$. But the latter is an abelian category, while $D^b_c(X,\CC)$ is not. A standard way to construct \textit{abelian} subcategories from a triangulated category such as $D^b_c(X,\CC)$ is to consider \textit{t-structures}; the conditions defining membership in $\mathrm{Perv}(X)$ is precisely the definition of the \textit{perverse t-structure}.\\

\noindent
\textbf{Remark.} We can equivalently formulate the condition of perversity as follows. Let $X = \bigcup_\alpha X_\alpha$ be a stratification of $X$ on which a sheaf $\mathscr{F}^\bullet$ is locally constant on each strata. Then $\mathscr{F}^\bullet$ is perverse if and only if
$$\mathscr{H}^k(i^*_\alpha \mathscr{F}^\bullet) = 0 \text{ for } k \not \in [-\mathrm{dim}\, X, -\mathrm{dim}\, X_\alpha].$$
$$\mathscr{H}^k_c(i^!_\alpha \mathscr{F}^\bullet) = 0 \text{ for } k \not \in [\mathrm{dim}\, X_\alpha, \mathrm{dim}\, X]$$
where $i_\alpha: X_\alpha \to X$ is the inclusion map. The equivalence of these two definitions can be found in Section 8.1 of \cite{Hotta-Takeuchi-Tanisaki}.

Now suppose $X$ admits an action of a connected algebraic group $G$, which we can visualize as
$$G \times X \underset{\mathrm{pr}_2}{\overset{a}{\rightrightarrows}} X$$ 
where $a$ is the action map and $\mathrm{pr}_2$ is projection on the second factor. We define a $G$-equivariant perverse sheaf on $X$ as a pair of datum $(\mathscr{F}, \theta)$, where $\mathscr{F}$ is a perverse sheaf on $X$, and $\theta$ is an isomorphism
$$\theta: a^* \mathscr{F} \overset{\sim}{\longrightarrow} \mathrm{pr}_2^* \, \mathscr{F}$$
respecting group multiplication. For example, $\theta_\mathrm{id}: (a^*\mathscr{F})_{\mathrm{id} \times X} \to (\mathrm{pr}_2^*\mathscr{F})_{\mathrm{id} \times X}$ is the identity map. With obvious morphisms, we can build the category of $G$-equivariant perverse sheaves, denoted $\widehat{\mathrm{Perv}_G}(X)$. It turns out that $G$-equivariance is not an additional structure, but a condition that certain perverse sheaves satisfy. This is made precise by the following lemma (Lemma A.1.2. \cite{Zhu}).

\begin{lem} \label{lem: equivariantCondition} The forgetful functor $\widehat{\mathrm{Perv_G}}(X) \to \mathrm{Perv}(X)$ is fully faithful, with essential image consisting of those perverse sheaves $\mathscr{F}$ such that $a^*\mathscr{F} \cong \mathrm{pr}_2^* \, \mathscr{F}$.
\end{lem}

Thus, we may replace the category of $G$-equivariant perverse sheaves with its essential image in $\mathrm{Perv}(X)$; we will denote this category by $\mathrm{Perv}_G(X)$. 
\\

The central object we need to introduce in order to state geometric Satake is the category of $L^+G$-equivariant perverse sheaves on the affine Grassmannian, which, as we've seen in the previous section, involves a \textit{pro-algebraic group} acting on an \textit{ind-scheme}. Luckily, under suitable conditions, the constructions we've discussed so far will easily ``pass to the limit", giving us a suitable notion of equivariant perverse sheaves on $\gr$. 

Let $X = \varinjlim X_i$ be an ind-scheme, with an action of a pro-algebraic group $K = \varprojlim K_i$. We assume that the stabilizer of each geometric point only has finitely many connected components, each $X_i$ is a $K$-stable closed subscheme of finite type over $\CC$, and the action of $K$ on each $X_i$ factors through an algebraic quotient $K_i$. These are all conditions obviously satisfied by the action of $L^+G$ on $\gr$. Note that if $X_i \subset X_j$ is a closed subscheme, then we can assume that $K_i$ is a quotient of $K_j$, and we have functors 
$$\mathrm{Perv}_{K_i}(X_i) \overset{\sim}{\longrightarrow} \mathrm{Perv}_{K_j}(X_i) \longrightarrow \mathrm{Perv}_{K_j}(X_j)$$
where the first equivalence of categories is implied by Lemma \ref{lem: equivariantCondition}. Finally, we can define the category of $K$-equivariant perverse sheaves on $X$ as
$$\mathrm{Perv}_K(X) \defeq \varinjlim \mathrm{Perv}_{K_i}(X_i)$$
\\

Geometric Satake states that there is an equivalence of \textit{symmetric monoidal categories}
$$\mathrm{Perv}_{L^+G}(\gr) \overset{\sim}{\longrightarrow} \mathrm{Rep}(\hat{G})$$
So we need to come up with a symmetric monoidal structure on $\mathrm{Perv}_{L^+G}(\gr)$. Recall that we have a convolution map 
$$m: \gr \times \gr \cong \gr \tilde{\times} \gr \to \gr$$
The following ``miraculous theorem" (Proposition 5.1.4 in \cite{Zhu}) is the essential ingredient needed to defining the tensor structure on $\mathrm{Perv}_{L^+G}(\gr)$:

\begin{thm}
Let $\mathscr{F}, \mathscr{G}$ be sheaves in $\mathrm{Perv}_{L^+G}(\gr)$. Then $m_!(\mathscr{F} \boxtimes \mathscr{G})$ is again perverse. 
\end{thm}

We define the \textit{convolution product} $\mathscr{F} \star \mathscr{G}$ on two perverse sheaves as $m_!(\mathscr{F} \boxtimes \mathscr{G})$.

\noindent
\textbf{Example.} Let's write down explicitly some examples of $L^+G$-equivariant perverse sheaves on $\gr$ by doing the most na\"ive construction following \cite{Notes}:

First, pick a $G(\mathcal{O})$-orbit $\mathrm{Gr}_\mu$, and start with the (shifted) constant sheaf $\CC[\mathrm{dim} \, \mathrm{Gr}_\mu]$ on $\mathrm{Gr}_\mu$. Then we inductively extend to the next strata as follows: let $j: \mathrm{Gr}_\mu \to \mathrm{Gr}_\lambda \cup \mathrm{Gr}_\mu$ be the inclusion map; consider the sheaf $j_*\CC[\mathrm{dim} \, \mathrm{Gr}_\mu]$ on $\mathrm{Gr}_\lambda \cup \mathrm{Gr}_\mu$. This sheaf is no longer perverse, so we truncate off anything in degrees strictly higher than $-\mathrm{dim}\, X_\beta-1$. We call this sheaf $\mathrm{IC}_\mu$, the \textit{intersection cohomology} of $\mathrm{Gr}_{\leq \mu}$. This explicit construction leaves the following crucial properties easy to check:

\begin{enumerate}
\item $\mathrm{IC}_\mu$ is supported on $\overline{\mathrm{Gr}_\mu} = \mathrm{Gr}_{\leq \mu}$, and
$$\mathrm{IC}_\mu|_{\mathrm{Gr}_\mu} \cong \CC[\mathrm{dim}\, \mathrm{Gr}_\mu]$$
In other words, $\mathrm{IC}_\mu$ is the constant sheaf over the ``smooth locus".
\item For some $\mathrm{Gr}_\lambda \subseteq \mathrm{Gr}_{\leq \mu}$ (i.e., for $\lambda \leq \mu$), we have
$$\mathscr{H}^*(i_\lambda^*\mathrm{IC}_\mu) \text{ is concentrated in degrees } [-\mathrm{dim}\, \mathrm{Gr}_\mu, -\mathrm{dim}\, \mathrm{Gr}_\lambda-1] \text{, and }$$
$$\mathscr{H}^*_c(i_\lambda^!\mathrm{IC}_\mu) \text{ is concentrated in degrees } [\mathrm{dim}\, \mathrm{Gr}_\lambda +1, \mathrm{dim}\, \mathrm{Gr}_\mu]$$
In particular, $\mathrm{IC}_\mu$ is perverse.
\end{enumerate}
Furthermore, one notices that the intersection cohomologies $\mathrm{IC}_\mu$ are \textit{Verdier self-dual}: $\mathbf{D}(\mathrm{IC}_\mu) \cong (\mathrm{IC}_\mu)^\vee$, where $\mathbf{D}$ denotes the Verdier duality functor and $\bullet^\vee$ is the standard linear duality. 
\\

The ``extend and truncate" construction given in the example may seem ad hoc, but these intersection cohomology sheaves are in fact the most important objects in $\mathrm{Perv}_{L^+G}(\gr)$. For example, they form a basis under the symmetric monoidal structure given by convolution:

\begin{thm}[Decomposition Theorem] Let $\mathrm{IC}_\lambda, \mathrm{IC}_\mu$ be two intersection cohomology sheaves on $\gr$. Then
$$\mathrm{IC}_\lambda \star \mathrm{IC}_\mu \cong \bigoplus_\nu \mathrm{IC}_\nu$$
where those $\nu$ that appear in the direct sum satisfy $\lambda+\mu \geq \nu$. In particular, $\mathrm{IC}_{\lambda+\mu}$ appears with multiplicity one.
\end{thm}
This is a necessary condition for the geometric Satake correspondence to hold, since $\mathrm{Rep}(\hat{G})$ is a semisimple category. 

We are finally ready to formulate the geometric Satake correspondence: the last ingredient we need is Tannakian formalism. Let us recall this briefly. A $\CC$-linear tensor category $\mathcal{C}$ is \textit{Tannakian} if it admits a fully faithful functor $F: \mathcal{C} \to \mathrm{Vect}_\CC$. By general abstract nonsense, we can identify the essential image $F(\mathcal{C})$ as the category of representations of the group $\mathrm{Aut}(F)$.

\begin{thm}[Geometric Satake]\label{thm:geomSat} The cohomology functor $H^*: \mathrm{Perv}_{L^+G}(\gr) \to \mathrm{Vect}_{\CC}$ gives $\mathrm{Perv}_{L^+G}(\gr)$ the structure of a Tannakian category. Furthermore, the Tannakian group $\mathrm{Aut}(H^*)$ is canonically identified with $\hat{G}$, the Langlands dual group of $G$.
\end{thm}

We conclude this introduction to geometric Satake by mentioning that Theorem \ref{thm:geomSat} should be viewed as a \textit{categorification} of the classical Satake isomorphism, which states that the Hecke algebra $\mathcal{H}(\gr)$ of compactly supported $G(\mathcal{O})$-equivariant functions on $\gr$ is isomorphic as rings with the representation ring of $\hat{G}$. The analogy comes from the case of non-Archimedian local fields: in the case of $F = \mathbb{F}_p((t)), \mathcal{O} = \mathbb{F}_p[[t]]$, one replaces the category $\mathrm{Perv}_{L^+G}(\gr)$ of complex-valued perverse sheaves with the category of $\ell$-adic perverse sheaves for some prime $\ell \neq p$. Theorem \ref{thm:geomSat} holds just as before:
$$H^*: \mathrm{Perv}_{L^+G}(\gr, \overline{\QQ_\ell}) \overset{\sim}{\longrightarrow} \mathrm{Rep}_{\overline{\mathbb{Q}_\ell}}(\hat{G})$$
By Grothendieck's function-sheaf dictionary, we can recover a function on $\gr$ by taking traces of Frobenius on cohomology. Thus, geometric Satake lifts an isomorphism of rings to an equivalence of categories. 

\subsection{Geometry of MV Cycles} \label{subsect: MVcycles} 
The geometric Satake correspondence as stated in Theorem \ref{thm:geomSat} is concise, but sweeps under the rug many natural correspondences. The most important one for us is the following fact which we will make precise in this section: the geometry of $\gr$ remembers the highest weight theory of $\hat{G}$. The following extraordinary result explains this statement:

\begin{thm}[Mirkovi\'c, Vilonen]
There is a canonical isomorphism
$$H^*(S_\lambda, \mathrm{IC}_\mu) \cong \CC[\mathrm{Irr}(S_\lambda \cap \mathrm{Gr}_{\leq \mu})]$$
where the right hand side denotes the formal $\CC$-span of irreducible components in $S_\lambda \cap \mathrm{Gr}_{\leq \mu}$. In fact, there is a natural isomorphism of functors
$$H^*(\bullet) \cong \bigoplus_\lambda H^*(S_\lambda, \bullet): \mathrm{Perv}_{L^+G}(\gr) \to \mathrm{Vect}_\CC$$
where the direct sum is indexed over the coweight lattice of $G$. Identifying the essential image of $H^*$ with $\mathrm{Rep}_{\hat{G}}$, $H^*(\mathrm{IC}_\mu)$ is precisely the highest weight representation with highest weight $\mu$, $V_\mu$. Furthermore, the above decomposition
$$H^*(\mathrm{IC}_\mu) = H^*(\mathrm{Gr}_{\leq \mu}, \mathrm{IC}_\mu) = \bigoplus_\lambda H^*(S_\lambda, \mathrm{IC}_\mu)$$
is the usual weight space decomposition.
\end{thm}

The following useful observation follows immediately.
\begin{cor}
The dimension of the $\lambda$-weight space inside $V_\mu$ is precisely $\#\mathrm{Irr}(S_\lambda \cap \mathrm{Gr}_{\leq \mu})$.
\end{cor}

Let's see this result in action with a couple of examples. 

\noindent
\textbf{Example.} Let $G = \mathrm{PGL}_3$, and $\mu$ the dominant coweight $(1,0,-1)$. By the example in Section \ref{subsect: affineGrass}, we have that
$$\mathrm{Gr}_{\leq \mu} = \left\{g \in G(F)/G(\mathcal{O}): \begin{matrix}\nu(g_{ij}) \geq -1,\\ \nu(2 \times 2\text{-minors}) \geq -1,\\ \nu(\mathrm{det}(g_{ij})) = 0  \end{matrix}\right\}$$
We compute some examples of intersections $S_\lambda \cap \mathrm{Gr}_{\leq \mu}$ for various $\lambda$. Let's first consider the semi-infinite orbit $S_\lambda$ where $\lambda = (0,0,0)$. By definition, $S_\lambda$ is the $U(F)$-orbit of the point $t^\lambda = \mathrm{id}$ in $\gr$; we can easily calculate the stabilizer of this action:
$$\mathrm{Stab}_{U(F)}(t^\lambda) = \{n \in U(F): t^\mu \cdot n \cdot t^{-\mu} \in G(\mathcal{O})\} = U(\mathcal{O})$$
Thus, we can calculate the intersection $S_\lambda \cap \mathrm{Gr}_{\leq \mu}$ as

\begin{align*}
S_\lambda \cap \mathrm{Gr}_{\leq \mu}
&= \left\{\begin{pmatrix}1&x&z\\ 0&1&y\\0&0&1\end{pmatrix} \in U(F)/U(\mathcal{O}): \begin{matrix}\nu(x), \nu(y), \nu(z) \geq -1,\\ \nu(xy-z) \geq -1  \end{matrix}\right\}\\
\end{align*}
From here we can break into two cases. First we consider when $\nu(x) = -1$; this forces $\nu(y) = 0$ (hence $y = 0$), and there are no extra conditions on $z$ besides $\nu(z) \geq -1$. More explicitly, we have
$$x = x_{-1}t^{-1} \quad y = 0 \quad z = z_{-1}t^{-1}$$
This gives us a copy of $\mathbb{G}_m \times \mathbb{A}^1$, with coordinates $(x_{-1}, z_{-1})$. In the other case, the roles of $x$ and $y$ are switched, and we have
$$x = 0 \quad y = y_{-1}t^{-1} \quad z = z_{-1}t^{-1}$$
which gives us another copy of $\mathbb{G}_m \times \mathbb{A}^1$, with coordinates $(y_{-1}, z_{-1})$. Thus, we see that there are two irreducible components in the intersection $S_\mu \cap \mathrm{Gr}_{\leq \lambda}$:
$$S_\mu \cap \mathrm{Gr}_{\leq \lambda} = \mathbb{G}_m \times \mathbb{A}^1 \cup \mathbb{G}_m \times \mathbb{A}^1$$

Using the correspondence given by geometric Satake, we can view this geometric fact as reflecting the representation theory of $\hat{G} = \mathrm{SL}_3$. Indeed, the highest weight representation $\hat{G}$ with weight $\mu = (1, 0, -1)$ is the adjoint representation on $\mathfrak{sl}_3(\CC)$; this representation contains a two dimensional $\lambda = (0,0,0)$-weight space spanned by
$$\begin{pmatrix}1&0&0\\0&-1&0\\0&0&0 \end{pmatrix} \text{  and  }  \begin{pmatrix}0&0&0\\0&1&0\\0&0&-1 \end{pmatrix}.$$

\noindent
\textbf{Example.} Let's calculate another example with $G = \mathrm{PGL}_3$ and $\mu = (1, 0, -1)$. Note that for a general coweight $\lambda$, we have
$$\mathrm{Stab}_{U(F)}(t^\lambda) = U(F) \cap t^\lambda \cdot G(\mathcal{O}) \cdot t^{-\lambda}$$
Following the previous example, we have a general strategy for calculating the intersections $S_\lambda \cap \mathrm{Gr}_{\leq \mu}$:
$$S_\lambda \cap \mathrm{Gr}_{\leq \mu} = \left\{n \cdot t^\lambda: \begin{matrix}n \in U(F)/U(F) \cap t^\lambda G(\mathcal{O}) t^{-\lambda}\\\nu(\text{entries}) \geq -1,\\ \nu(2 \times 2\text{-minors}) \geq -1, \\ \nu(\mathrm{det}(n \cdot t^\lambda)) = 0. \end{matrix}\right\}$$
Let's calculate this in the case of $\lambda = (1, 0, -1)$. In this case, some $n \in U(F)$ stabilizes $t^\lambda$ if and only if
$$n = \begin{pmatrix}1&x&z\\0&1&y\\0&0&1\end{pmatrix} \in t^\lambda G(\mathcal{O}) t^{-\lambda} \iff \begin{pmatrix}1&t^{-1}x&t^{-2}z\\0&1&t^{-1}y\\0&0&1\end{pmatrix} \in G(\mathcal{O}) \iff x, y \in t\mathcal{O}, z \in t^2\mathcal{O}$$
So we can consider $x,y$ as elements in $\mathcal{O}/t\mathcal{O}$ and $z$ as an element in $\mathcal{O}/t^2\mathcal{O}$. Now by definition, we have
$$S_\lambda = \begin{pmatrix}1&x&z\\0&1&y\\0&0&1\end{pmatrix} \cdot t^\lambda = \begin{pmatrix}t&x&t^{-1}z\\0&1&t^{-1}y\\0&0&t^{-1}\end{pmatrix}$$
On the other hand, the criteria on valuations defining $\mathrm{Gr}_{\leq \mu}$ forces $x, y, z$ to have valuation at least -1, and there are no further conditions. So we can write
$$x = x_0 \quad y = y_0 \quad z = z_0 + z_1t$$
for arbitrary choices of $x_0, y_0, z_0, z_1$. Thus, the intersection $S_\lambda \cap \mathrm{Gr}_{\leq \mu}$ is irreducible and isomorphic to $\mathbb{A}^4$, with coordinates $(x_0, y_0, z_0, z_1)$. 

From the point of view of the representation theory of $\hat{G} = \mathrm{SL}_3$, the unique MV cycle in this intersection corresponds to the one dimensional highest weight space in the adjoint representation on $\mathfrak{sl}_3(\CC)$ spanned by
$$\begin{pmatrix}0&0&1\\0&0&0\\0&0&0 \end{pmatrix}$$

We conclude this introduction to the geometric Satake correspondence by mentioning the following well-known result, which will allow us to assume without loss of generality that $G$ is simply connected for certain calculations involving MV cycles. 

\begin{prop} \label{prop: simplyConnGr}Let $\g$ be a semisimple Lie algebra, and $\tilde{G}$ the unique simply connected Lie group with Lie algebra $\g$. Let $G$ be another connected Lie group with Lie algebra $\g$, then we have a noncanonical isomorphism
$$\mathrm{Gr}_G \cong \bigsqcup_{\gamma \in \pi_1^{\mathrm{alg}}(G)} \mathrm{Gr}_{\tilde{G}}$$
where the disjoint union on the right hand side is indexed over the algebraic fundamental group of $G$.
\end{prop}

\section{MV Polytopes and Quivers}
Now that we've motivated and defined MV cycles, one issue becomes apparent: MV cycles are difficult to calculate in general. Fortunately, the work of Anderson \cite{Anderson} and Baumann-Kamnitzer \cite{Baumann-Kamnitzer} provides combinatorial models for these cycles in the form of polytope geometry. In fact, Kamnitzer shows in \cite{Kamnitzer} that the two combinatorial models are equivalent. In this section, we will define the two approaches to arrive at MV polytopes, and demonstrate their equivalence by concrete examples.

\subsection{The Moment Map} \label{subsect: momentMap}

We briefly recall the definition of the moment map; this can be found in any introductory texts on sympletic geometry. In general, the moment map is defined for a Lie group $G$ acting on a symplectic manifold $M$, where the action is Hamiltonian. Under these conditions, the moment map is a $G$-equivariant map
$$\mu: M \to \mathrm{Lie}(G)^*$$
satisfying some condition. For our purposes, however, it will suffice to only consider actions of tori on projective varieties, which simplifies the situation significantly. Thus, we will only describe the moment map in these specific cases. 

Let $T_\CC$ be a complex torus, and $T$ a maximal compact subgroup of $T_\CC$. Let $(V, (\cdot , \cdot))$ be a complex vector space with a Hermitian form, along with an action of the complex torus $T_\CC$, such that $T$ preserves the Hermitian form. The moment map $\mu$ for the action of $T$ on $V$ is a map $\mu: V \to \mathfrak{t}^*= \mathrm{Lie}(T)^*$ satisfying
$$\langle \mu(v), X \rangle = -\frac{i}{2}(X \cdot v, v)$$
for all $X \in \mathfrak{t} = \mathrm{Lie}(T)$ and $v \in V$. The angled brackets here $\langle \cdot, \cdot \rangle$ denote the canonical pairing on $\mathfrak{t}$ and $\mathfrak{t}^*$. 

Note that under the conditions specified here, we can always assume there is an orthonormal eigenbasis $e_1, \ldots, e_n$ for the action of $T_\CC$. We will denote by $\lambda_k \in \mathfrak{t}^*$ the infinitesimal weight defined by
$$X \cdot e_k = i \langle \lambda_k, X\rangle e_k$$
for every $X \in \mathfrak{t}$. The action of the Lie algebra $\mathfrak{t}$ on $V$ is the obvious one given by the differential of the action of $T$ on $V$. If we write some $v$ in $V$ in terms of the orthonormal basis $v = \sum_{k=1}^n z_ke_k$, we can explicitly describe the moment map as
$$\mu(v) = \frac{1}{2} \sum_{k=1}^n |z_k|^2\lambda_k.$$

Let $U = U(V)$ be the group of automorphisms of $V$ preserving the Hermitian form. Recall that $\mathbb{P}(V)$ has a natural symplectic form $\omega$, for which the action of $U$ is Hamiltonian. 
In this case, the moment map $\mu: \mathbb{P}(V) \to \mathrm{Lie}(U)^*$ is defined by the following condition:
$$\langle \mu([v]), X \rangle = -\frac{i}{2}\frac{(X \cdot v, v)}{\norm{v}^2}$$

Now suppose $X \subseteq \mathbb{P}(V)$ is an irreducible projective variety, stable under the action of the compact torus $T \subseteq U$. The crucial and well known result we need is the following:

\begin{thm}\label{thm: momentMap} Let $X \subseteq \mathbb{P}(V)$ be a closed irreducible subvariety, stable under the action of $T$. Then the image of the fixed points $X^T$ under $\mu$ is a finite subset of $\mathfrak{t}^*$. Furthermore, we have
$$\mathrm{conv}\left(\mu\left(X^T\right)\right) = \mu(X)$$
\end{thm}

In other words, to calculate the moment map image of a torus invariant projective variety, it suffices to calculate the image of the fixed points. 

We now specialize to the case of MV cycles. Let $G$ be a complex reductive algebraic group, $T \subset G$ a maximal torus, and $\hat{T} \subset \hat{G}$ the Langlands dual torus contained in the Langlands dual group of $G$. Recall that the MV cycles $S_\lambda \cap \mathrm{Gr}_{\leq \mu} \subseteq \gr$ are quasi-projective varieties, invariant under the action of $T$. In fact, we can embed $\gr$ in a projective space $\mathbb{P}(V)$ via an ample line bundle; we will take this for granted (for example, for $G = \mathrm{GL}_n$ we can use the determinant line bundle; see Section 1.5 of \cite{Zhu}). Choosing a Hermitian product on $V$ invariant under the maximal compact subgroup of $T$, for every MV cycle $S_\lambda \cap \mathrm{Gr}_{\leq \mu}$ we obtain a moment map
$$\mu: S_\lambda \cap \mathrm{Gr}_{\leq \mu} \to \mathrm{Lie}(T)^*$$
Using the Killing form to identify $\mathrm{Lie}(T)^*$ with $\mathrm{Lie}(T)$, which is again canonically identified with $\mathrm{Lie}(\hat{T})^*$, we view the moment map as a map from an MV cycle to $\mathrm{Lie}(\hat{T})^*$. We caution the reader here that the traditional definition of the moment map given above will differ by a scaling constant from Anderson's conventions in \cite{Anderson}, and we will follow the latter. In particular, the moment map image will always land in the real subspace $\mathrm{Lie}(\hat{T})^*_\RR$. 

\begin{defn} We say that a convex polytope $P$ in $\mathrm{Lie}(\hat{T})_\RR^*$ is an MV polytope if it is the image of the closure of an MV cycle in $\gr$ under the moment map.
\end{defn}

We can describe $\mu$ more explicitly. First, decompose $V$ into weight spaces for the action of $T$
$$V = \bigoplus_{\lambda \in \mathbf{X}^*(T)} V_\lambda$$
We can without loss of generality assume that this decomposition is orthogonal, with respect to the Hermitian product we picked on $V$. Pick an orthonormal eigenbasis $\{v_\lambda\}_{\lambda \in \mathbf{X}^*(T)}$; then any $v$ in $V$ can be written as a finite sum $v = \sum_{\lambda \in \mathbf{X}^*(T)} c_\lambda v_\lambda$. Note that we can rewrite the right hand side of the moment map on projective space (up to a scaling constant) as
$$ \frac{(X \cdot v, v)}{\norm{v}^2} = \frac{1}{\norm{v}^2} \sum_\lambda \left(c_\lambda \langle \lambda, X \rangle, v\right) = \frac{1}{\norm{v}^2}\sum_\lambda \left\langle  |c_\lambda|^2  \lambda, X \right\rangle$$
for any $X$ in $\mathrm{Lie}(T)$. Here we adopt the slight abuse of notation by denoting both a character $\lambda \in \mathbf{X}^*(T)$ and its normalized derivative $\frac{1}{2\pi i} d\lambda$ (hence an element of $\mathrm{Lie}(T)^*$) by the same symbol $\lambda$. So in coordinates provided by the weight space decomposition of $V$, we have that
$$\mu([v]) = \sum_\lambda \frac{|c_\lambda|^2}{\norm{v}^2} \lambda$$
The following proposition will help us determine the image of an MV cycle under the moment map.

\begin{prop} \label{prop: momentMap}The fixed points of the torus action on $\gr$ are precisely the collection of $t^\lambda$ for coweights $\lambda \in \mathbf{X}_*(T)$. For these fixed points, we have
$$\mu: t^\lambda \mapsto \lambda$$
where on the right hand side we view $\lambda$ as weights in $\mathrm{Lie}(\hat{T})^*$.
\end{prop}

The following result follows immediately from Theorem \ref{thm: momentMap} and Proposition \ref{prop: momentMap}:

\begin{cor} Let $\overline{S_\lambda \cap \mathrm{Gr}_{\leq \mu}}$ be the closure of an MV cycle in $\gr$. Then the associated moment polytope is
$$\mathrm{conv}(\nu: t^\nu \in \overline{S_\lambda \cap \mathrm{Gr}_{\leq \mu}})$$
where the convex hull is taken in $\mathrm{Lie}(\hat{T})^*$, interpreting $\nu$ as a weight of $\mathrm{Lie}(\hat{G})$.
\end{cor}

\noindent
\textbf{Example.} Let $G = \mathrm{PGL}_3$, and $\hat{G} = \mathrm{SL}_3$. Anderson \cite{Anderson} describes the MV polytopes associated to $G$ by providing the following generators $\alpha_1, \alpha_2, \beta_1, \beta_2$:
\\

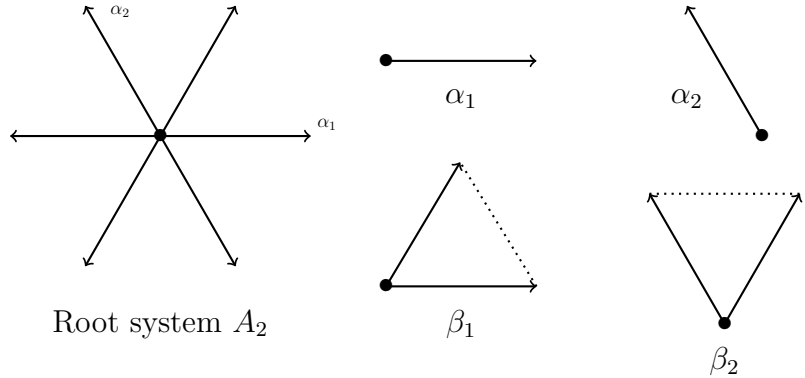
\begin{center}
\begin{tikzpicture}
\foreach\ang in {60,120,...,360}{
     \draw [thick] [->] (0,0) -- (\ang:2cm);
    }
\node[anchor=south west,scale=0.6] at (2,0) {$\alpha_1$};
\node[anchor=south west,scale=0.6] at (-0.75, 1.5) {$\alpha_2$};
\node at (0, -2.5) {Root system $A_2$};
\node at (0,0) {$\bullet$};

\draw [thick] [->]  (3,1) -- (5,1);
\node at (3,1) {$\bullet$};
\node at (4, 0.5) {$\alpha_1$};

\node at (8, 0) {$\bullet$};
\draw [thick]  [->] (8,0) -- (7, 1.732);
\node at (7, 0.5) {$\alpha_2$};

\draw [thick]  [->] (3,-2) -- (355:4cm);
\draw [thick] [dotted]  (5,-2) -- (355:4cm);

\draw [thick] [->] (3,-2) -- (5,-2);
\node at (3,-2) {$\bullet$};
\node at (4, -2.5) {$\beta_1$};

\begin{scope}[shift={(7.5,-2.5)}]
{
\foreach\ang in {60,120}{
     \draw [thick] [->] (0,0) -- (\ang:2cm);
    }
 \node at (0,0) {$\bullet$};
 \node at (0,-0.5) {$\beta_2$};
 \draw [thick] [dotted] (120:2cm) -- (60:2cm);
 }   
  \end{scope}

\end{tikzpicture}
\end{center}
These ``primitive" MV polytopes generate all other MV polytopes by Minkowski sums:
$$P +_M Q := \{p+q: p \in P, q \in Q\}$$
Notice that there is one ``relation" among the primitive polytopes:
$$\alpha_1 +_M \alpha_2 = \beta_1 \cup \beta_2$$

To see that the primitive MV polytopes $\beta_1, \beta_2$ arise as moment map images of MV cycles on $\gr$, consider the coweight
$$\mu: t \mapsto \begin{pmatrix}t&0&0\\ 0&1&0\\0&0&t^{-1} \end{pmatrix}$$
of $G$. Recall by the examples calculated in Section \ref{subsect: MVcycles}, we have two MV cycles given by the following intersection:

\begin{align*}
    S_0 \cap \mathrm{Gr}_{\leq \mu} &= \left\{\begin{pmatrix}1&x&y\\0&1&z\\0&0&1\end{pmatrix} \in \gr: \begin{matrix}x = x_{-1}t^{-1} \neq 0, y = 0, z = z_{-1}t^{-1} \text{ or }\\ x = 0, y = y_{-1}t^{-1} \neq 0, z = z_{-1}t^{-1}\end{matrix}\right\}\\
    &= \mathbb{G}_m \times \mathbb{A}^1 \cup \mathbb{G}_m \times \mathbb{A}^1
\end{align*}

Let's take the first MV cycle $X = \mathbb{G}_m \times \mathbb{A}^1$, with coordinates $(x_{-1}, z_{-1})$. In the closure of $X$, we see the coweights $t^0 = \mathrm{id}, t^{(1,-1,0)}, t^\mu$; this gives the MV polytopes $\beta_1$. Analogously, the second MV cycle $Y = \mathbb{G}_m \times \mathbb{A}^1$, with coordinates $(y_{-1}, z_{-1})$, gives the MV polytope $\beta_2$.

\subsection{The Preprojective Algebra} While MV polytopes are conceptually simple and easy to define, it is still difficult to compute them without knowing the MV cycles themselves. In this section, we present a completely different and significantly more tractable way to construct MV polytopes: via quiver representations. Recall that a \textit{quiver} $Q$ is a finite oriented graph, given by a pair of datum $(Q_0, Q_1)$, where $Q_0$ is the set of vertices and $Q_1$ is the set of arrows. We have two functions $s, t: Q_1 \to Q_0$, taking an arrow to its source and target, respectively. 

\begin{defn} Let $Q = (Q_0, Q_1)$ be a quiver. A representation $V$ of $Q$ is a collection of vector spaces $V_i$ for each $i$ in $Q_0$, and a collection of linear maps $V_\alpha$ for each $\alpha$ in $Q_1$, such that
$$V_\alpha: V_{s(\alpha)} \to V_{t(\alpha)}$$
The dimension vector $\underline{\mathrm{dim}}(V)$ of a representation $V$ is the $Q_0$-tuple $(\mathrm{dim}_k V_i)_{i \in Q_0} \in \NN^{Q_0}$.
\end{defn}

Let $V$ and $W$ be two representations of the quiver $Q$. A morphism between $V$ and $W$ is then a collection of linear maps $(\varphi_i: V_i \to W_i)_{i \in Q_0}$ such that for every arrow $\alpha$ between two vertices $i,j$, the following square commutes:
$$
\begin{CD}
V_i @>V_\alpha>> V_j\\
@V\varphi_i VV @VV\varphi_jV\\
W_i @>W_\alpha>> W_j
\end{CD}
$$

Equivalently, we can think of quiver representations of $Q$ as left modules over the \textit{path algebra} $kQ$, where $kQ$ is the noncommutative ring generated by the symbols $\{e_i, \alpha\}_{i \in Q_0, \alpha \in Q_1}$ subject to the relations
$$e_ie_j = \delta_{i,j}e_i \quad \alpha e_i = \delta_{s(\alpha),i}\alpha \quad e_j\alpha = \delta_{t(\alpha), j}\alpha$$
and the product of two arrows $\alpha\beta$ is zero unless $t(\beta) = s(\alpha)$. One can easily see that the category of quiver representations and the category of finite dimensional left $kQ$-modules are equivalent by the following identifications:
\begin{align*}
\mathrm{Rep}(Q) & \stackrel{\cong}{\longrightarrow} kQ\text{-}\mathrm{Mod}\\
(V_i)_{i \in Q_0} & \longmapsto V = \bigoplus_{i \in Q_0} V_i
\end{align*}
where $e_i$ acts on $V$ by projection onto $V_i$, and an arrow $\alpha$ acts on $V$ by $V_\alpha \circ e_i$. 

Let $Q$ be a quiver. We define the double quiver $Q^* = (Q^*_0, Q^*_1)$ to have the same set of vertices $Q_0^* = Q_0$, but double the amount of edges
$$Q_1^* = Q_1 \sqcup \overline{Q}_1$$
where 
$$\overline{Q}_1 = \{\alpha^*: \alpha \in Q_1\}, \text{  with  } s(\alpha^*) = t(\alpha) \text{  and  } t(\alpha^*) = s(\alpha)$$
In other words, for every arrow $\alpha$ in $Q_1$, we append an arrow $\alpha^*$ going in the opposite direction. Let the \textit{preprojective algebra} of $Q$, denoted $\Pi Q$, be the quotient of $kQ^*$ by the two sided ideal generated by the following \textit{preprojective relation}:
$$\sum_{\alpha \in Q_1} \alpha\alpha^* - \alpha^*\alpha$$

In \cite{Baumann-Kamnitzer-Tingley} and \cite{Baumann-Kamnitzer}, Tingley, Baumann and Kamnitzer showed that modules over $\Pi Q$ are intimately connected with MV polytopes in the case when $Q$ is Dynkin (that is, the underlying graph of $Q$ is a Dynkin diagram). In fact, \cite{Baumann-Kamnitzer-Tingley} extends this to the affine case, but we will only worry about Dynkin type quivers. 
\\

\noindent
\textbf{Example.} (The Maya modules). This is an example of a family of $\Pi Q$-modules given in \cite{Kamnitzer-Sadanand}. Consider the type $A_n$ quiver $Q$ oriented as follows:
\begin{align*}
\bullet \longleftarrow \bullet \longleftarrow \cdots \longleftarrow \bullet \\
\hspace{0.5cm} 1 \hspace{1cm} 2 \hspace{0.8cm} \cdots  \hspace{0.8cm} n
\end{align*}
Its preprojective algebra $\Pi Q$ has two types of arrows:
$$\alpha_i: i \to i-1 \text{  for  } 1 < i \leq n$$
$$\alpha_i^*: i-1 \to i \text{  for  } 1 < i \leq n$$
For every proper subset $A$ of $\{1, \ldots, n\}$ of size $m$ other than $\{1, \ldots, m\}$, we define a Maya module $N(A)$ over $\Pi Q$. Let $A = \{a_1 < \cdots < a_m\}$, then $N(A)$ has basis
$$\bigcup_{1 \leq k \leq m}\{v_{k,k}, v_{k+1, k}, \ldots, v_{a_k-1, k}\}$$ 
where $v_{j,k}$ is an element in $N(A)_j$. We define the actions of the arrows $\alpha_j$ and $\alpha_j^*$ as follows
$$\alpha_j: v_{j,k} \mapsto v_{j-1,k}$$
$$\alpha_j^*: v_{j,k} \mapsto v_{j+1, k+1}$$

For an explicit example, consider the $A_5$ quiver
\begin{align*}
\bullet \longleftarrow \bullet \longleftarrow \bullet \longleftarrow \bullet \longleftarrow \bullet\\
 1 \hspace{0.92cm} 2 \hspace{0.92cm} 3\hspace{0.92cm}4\hspace{0.92cm}5
\end{align*}
along with the subset $A = \{2, 4, 5\} \subset \{1, \ldots, 5\}$. Then the Maya module $N(A)$ is 5 dimensional, with basis vectors
$$v_{1,1}, v_{2,2}, v_{3,2}, v_{3,3}, v_{4,3}$$
We can visualize the module structure of $N(A)$ by the following diagram:
\\
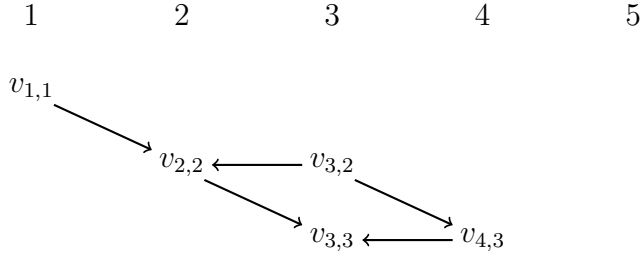
\begin{center}
\begin{tikzpicture}
\node at (0,0) {1};
\node at (2,0) {2};
\node at (4,0) {3};
\node at (6,0) {4};
\node at (8,0) {5};
\node at (0,-1) {$v_{1,1}$};
\node at (2,-2) {$v_{2,2}$};
\node at (4, -2) {$v_{3,2}$};
\node at (4,-3) {$v_{3,3}$};
\node at (6,-3) {$v_{4,3}$};
\draw [thick] [->] (0.3,-1.2) -> (1.6,-1.8);
\draw [thick] [<-] (2.4,-2) -> (3.6,-2);
\draw [thick] [->] (2.3,-2.2) -> (3.6,-2.8);
\draw [thick] [->] (4.3,-2.2) -> (5.6,-2.8);
\draw [thick] [<-] (4.4,-3) -> (5.6,-3);
\end{tikzpicture}
\end{center}
As we can see, the the preprojective relation is satisfied at the parallelogram formed by $v_{3,2}, v_{2,2}, v_{3,3}, v_{4,3}$, since
$$(\alpha_4\alpha_4^*-\alpha_3^*\alpha_3)(v_{3,2}) = 0$$

Note that if we fix a dimension vector $\nu \in \NN^{Q_0}$, then the datum of a $\Pi Q$-module is given by a collection of linear maps between finite dimensional vector spaces. We can thus regard a $\Pi Q$-module as a point in some algebraic variety $\Lambda(\nu)$ known as Lusztig's nilpotent variety. The following crucial theorem is due to Kamnitzer and Baumann \cite{Baumann-Kamnitzer}:

\begin{thm}[Baumann, Kamnitzer] \label{thm: quiverToPolytope} Let $Q$ be a Dynkin quiver. Let $Z$ be an irreducible component of $\Lambda(\nu)$ for some dimension vector $\nu$. Then for a dense open subset $U$ of $Z$, the following function is constant:
$$\mathrm{Pol}: Z \to \RR^{Q_0}$$
$$V \mapsto \mathrm{conv}\left(\left\{\underline{\mathrm{dim}}(V'): V' \subseteq V\right\}\right)$$
Furthermore, identifying $\RR^{Q_0}$ with $\mathrm{Lie}(\hat{T})^*_\RR$ by identifying $Q_0$ with simple roots, this generic value $\mathrm{Pol}(Z)$ is an MV polytope. 
\end{thm}

It turns out that the identification $Z \mapsto \mathrm{Pol}(Z)$ is a bijection between the irreducible components of Lusztig's nilpotent variety $\mathfrak{B} := \bigcup_{\nu \in \NN^{Q_0}} \Lambda(\nu)$ and the set $\mathcal{MV}$ of MV polytopes. In fact, this is a manifestation of several isomorphisms of \textit{crystal structures}. Thus, we defer this discussion until crystals are introduced.
\\

\noindent
\textbf{Remark.} Let $G_1, G_2$ be two complex reductive algebraic groups with the same Lie algebra $\g$. Let $\hat{G}_1, \hat{G}_2$ be the Langlands dual groups of $G_1$ and $G_2$, respectively, and let $\hat{\g} = \mathrm{Lie}(\hat{G}_1) = \mathrm{Lie}(\hat{G}_2)$. Note that in Theorem \ref{thm: quiverToPolytope}, the choice of $Q$ depends only on $\hat{\g}$, whereas the resulting MV polytopes seem like they should depend on the choice of $G_1$ or $G_2$. This is not a contradiction, as the MV polytopes obtained from $\mathrm{Gr}_{G_1}$ and $\mathrm{Gr}_{G_2}$ are the same up to translation. This is a direct consequence of Proposition \ref{prop: simplyConnGr}.
\\

We end this section with some low rank calculations justifying the second statement in Theorem \ref{thm: quiverToPolytope}; i.e., that we can construct MV polytopes via modules over the preprojective algebra. 
\\

\noindent
\textbf{Example.} Recall the primitive MV polytopes $\alpha_1, \alpha_2, \beta_1, \beta_2$ for $\hat{G} = \mathrm{SL}_3$ we calculated in Section \ref{subsect: momentMap}. Theorem \ref{thm: quiverToPolytope} guarantees that we can realize these MV polytopes as dimension vectors of (submodules of) $\Pi Q$-modules. In this case, we have $Q$ is the $A_2$ quiver
\begin{align*}
\bullet \stackrel{\alpha}{\longleftarrow} \bullet\\
1 \hspace{0.95cm} 2
\end{align*}
Consider the modules given by the following diagrams:\\

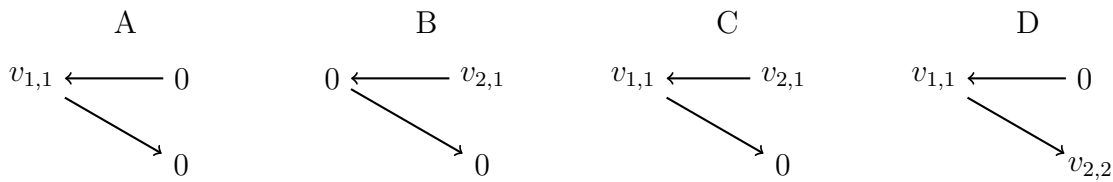
\begin{center}
\begin{tikzpicture}
\node at (1.25,0.75) {A};
\node at (0,0) (a) {$v_{1,1}$};
\node at (2,0) (b) {0};
\draw [thick] [<-] (a.0) -- (b.180);
\draw [thick] [->] (a) -- (330:2cm);
\node at (330:2.3cm) {0};

\begin{scope}[shift={(4,0)}]
{
\node at (1.25,0.75) {B};
\node at (0,0) (a) {0};
\node at (2,0) (b) {$v_{2,1}$};
\draw [thick] [<-] (a.0) -- (b.180);
\draw [thick] [->] (a) -- (330:2cm);
\node at (330:2.3cm) {0};
 }   
  \end{scope}
  
\begin{scope}[shift={(8,0)}]
{
\node at (1.25,0.75) {C};
\node at (0,0) (a) {$v_{1,1}$};
\node at (2,0) (b) {$v_{2,1}$};
\draw [thick] [<-] (a.0) -- (b.180);
\draw [thick] [->] (a) -- (330:2cm);
\node at (330:2.3cm) {0};
 }   
  \end{scope}

\begin{scope}[shift={(12,0)}]
{
\node at (1.25,0.75) {D};
\node at (0,0) (a) {$v_{1,1}$};
\node at (2,0) (b) {0};
\draw [thick] [<-] (a.0) -- (b.180);
\draw [thick] [->] (a) -- (330:2cm);
\node at (330:2.4cm) {$v_{2,2}$};
 }   
  \end{scope}
\end{tikzpicture}
\end{center}
where $\alpha$ and $\alpha^*$ acts as defined in the case of Maya modules. We can easily list the dimension vectors of each of their submodules:

\begin{center}
\text{Module A}: (0,0), (1,0)\\
\text{Module B}: (0,0), (0,1)\\
\text{Module C}: (0,0), (1,0), (1,1)\\
\text{Module D}: (0,0), (1,1), (0,1)

\end{center}
This indeed gives us all the primitive MV polytopes: Module A gives $\alpha_1$, Module B gives $\alpha_2$, Module C gives $\beta_1$, and Module D gives $\beta_2$.

Recall that these primitive MV polytopes have the relation
$$\alpha_1 +_M \alpha_2 = \beta_1 \cup \beta_2$$
This is reflected by the fact that
$$C \in \mathrm{Ext}^1_{\Pi Q}(A, B) \, \text{ and } \, D \in \mathrm{Ext}^1_{\Pi Q}(B, A)$$
span the respective one dimensional extension groups. All other $\Pi Q$-modules are simply direct sums of $A, B, C, D$, and we can easily see that
$$\mathrm{Pol}(M \oplus N) = \mathrm{Pol}(M) +_M \mathrm{Pol}(N)$$
for any $\Pi Q$-module $M$ and $N$. This agrees with Anderson's conclusion that $\alpha_1, \alpha_2, \beta_1, \beta_2$ generate all other MV polytopes via Minkowski sum.

\section{Galleries} In this section, we introduce another combinatorial model for MV cycles developed by Gaussent and Littelmann, known as galleries \cite{Gaussent-Littelmann}. As this model was developed independently of the polytope-geometric methods described in the previous section, the connection between the two models remains largely unexplored. The first steps towards a direct connection between the two models is carried out by Baumann and Gaussent by comparing the crystal structures on galleries with the crystal structures on MV cycles as defined by Braverman, Finkelberg, and Gaitsgory \cite{Baumann-Gaussent}. We will discuss these crystals in more detail in the following sections. 

\subsection{The Coxeter Complex $\mathscr{A}^{\mathrm{aff}}$} LS galleries are conceptually simple, but require some definitional setup. Thus, we will recall here the vocabulary and conventional notations associated to Coxeter complexes, which we will need to define galleries. Readers familiar with the theory of Bruhat-Tits buildings are encouraged to skim through this section for notation and proceed to Section \ref{subsect: LSgalleries}.

We begin by defining the affine root system and the affine Weyl group. Let $G$ be a complex connected reductive algebraic group, and $T \subset B \subset G$ a maximal torus contained in a Borel. Let $\Lambda$ be the cocharacter lattice of $T$, $\Lambda_\RR = \Lambda \otimes_\ZZ \RR$, $\Phi$ be the roots of $G$, and $\{\alpha_i: i \in I\}$ a collection of simple roots. By an \textit{affine root}, we mean an element of $\Phi^{\mathrm{aff}} = \Phi \times \ZZ$. As in the ``spherical" case, associated to every affine root $(\alpha, n)$ we have an affine reflection 
$$s_{\alpha, n}: \Lambda_\RR \to \Lambda_\RR$$
$$x \mapsto x - (\langle\alpha, x\rangle - n)\alpha^\vee$$
In other words, $s_{\alpha, n}$ is the reflection across the affine hyperplane
$$H_{\alpha, n} = \{x \in \Lambda_\RR: \langle \alpha, x\rangle = n\}$$
The affine Weyl group $W^{\mathrm{aff}}$ is the group generated by these affine reflections. The hyperplane arrangement formed by $\{H_\beta: \beta \in \Phi^{\mathrm{aff}}\}$ divides $\Lambda_\RR$ into faces. Faces with maximal dimensions are called \textit{alcoves}, faces of codimension 1 are called \textit{facets}, and faces of dimension 0 are called \textit{vertices}. We will denote by $\mathscr{A}^{\mathrm{aff}}$ the space $\Lambda_\RR$ along with its polysimplicial structure; evidently, $\mathscr{A}^{\mathrm{aff}}$ admits an action of $W^{\mathrm{aff}}$. We also fix the following notation for halfspaces determined by a hyperplane:
$$H^-_\beta = \{x \in \Lambda_\RR: \langle\beta,x\rangle \leq 0\} \text{ for } \beta \in \Phi^{\mathrm{aff}}$$

We define the dominant Weyl chamber $C_\mathrm{dom}$ and the fundamental alcove $A_\mathrm{fund}$ as follows:
$$C_\mathrm{dom} \defeq \{x \in \Lambda_\RR: \langle \alpha_i, x\rangle \geq 0 \text{ for all } i\}$$
$$A_\mathrm{fund} \defeq \{x \in C_\mathrm{dom}: \langle \theta, x\rangle \leq 1\}$$
where $\theta$ is the highest root. 

Let $F$ be an arbitrary face in $\mathscr{A}^{\mathrm{aff}}$; then there exists a unique face $F'$ of the fundamental alcove such that $w \cdot F = F'$ for some element $w$ in the affine Weyl group. We define the \textit{type} of an arbitrary face $F$ to be this unique face $F'$ of $A_{\mathrm{fund}}$.
\\

\noindent
\textbf{Example.} Here are the pictures of $\mathscr{A}^{\mathrm{aff}}$ for type $A_1$ and $A_2$.

Type $A_1$: Let $G = \mathrm{SL}_2$. Then the cocharacter lattice of the standard torus is one dimensional:
$$\Lambda_\RR = \ZZ\langle \alpha^\vee \rangle \otimes_\ZZ \RR$$
generated by the coroot $\alpha^\vee: t \mapsto \mathrm{diag}(t, t^{-1})$. The affine root system in this case is
$$\Phi^{\mathrm{aff}} = \{\pm \alpha\} \times \ZZ$$
where $\alpha$ is the character $\mathrm{diag}(a,b) \mapsto ab^{-1}$. The affine Weyl group is generated by reflections across the hyperplanes
$$H_{\alpha, n} = \{c\alpha^\vee \in \Lambda_\RR: 2c = n\} $$
for all $n \in \ZZ$. In pictures, we have
\\

\begin{center}
\begin{tikzpicture}
\node at (-5, 0) {$\mathscr{A}^{\mathrm{aff}} = $};
\draw (-4, 0) -- (4,0);
\node at (-3, 0) {$\circ$};
\node at (-2, 0) {$\bullet$};
\node at (-1, 0) {$\circ$};
\node at (0, 0) {$\bullet$};
\node at (1, 0) {$\circ$};
\node at (2, 0) {$\bullet$};
\node at (3, 0) {$\circ$};
\node at (-2, -0.5) {$-\alpha^\vee$};
\node at (0, -0.5) {$0$};
\node at (2, -0.5) {$\alpha^\vee$};
\node at (0,0) {$($};
\node at (1,0) {$)$};
\node at (0.5, 0.5) {$A_{\mathrm{fund}}$};
\end{tikzpicture}
\end{center}
where the action of $W^{\mathrm{aff}}$ is given by reflection about the labeled points $\left\{\frac{n}{2}\alpha^\vee: n \in \ZZ\right\}$. The fundamental chamber is
$$C_{\mathrm{dom}} = \{c\alpha^\vee: c \geq 0\}$$
and the fundamental alcove $A_{\mathrm{fund}}$ as labeled in the picture. The two types of facets are labeled by black dots and white dots.
\\

Type $A_2$: Let $G = \mathrm{SL}_3$. Then the cocharacter lattice of the standard torus is two-dimensional:
$$\Lambda_\RR = \ZZ\langle \alpha_1^\vee, \alpha_2^\vee\rangle \otimes_\ZZ \RR$$
generated by the standard coroots $\alpha_1, \alpha_2$. The affine root system in this case is
$$\Phi^{\mathrm{aff}} = \{\pm \alpha_1, \pm \alpha_2, \pm (\alpha_1 +\alpha_2)\} \times \ZZ$$
In the below picture of $\mathscr{A}^{\mathrm{aff}}$, the dashed lines represent several hyperplanes with respect to which the action of $W^{\mathrm{aff}}$ is generated:
\\

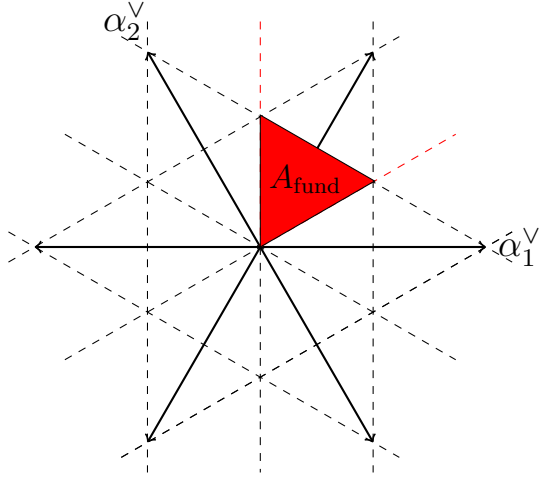
\begin{center}
\begin{tikzpicture}[scale=1.5,every node/.style={scale=1.5}]
\foreach\ang in {60,120,...,360}{
     \draw [thick] [->] (0,0) -- (\ang:2cm);
    }
\node[scale=0.8] at (2.3,0) {$\alpha_1^\vee$};
\node[scale=0.8] at (-1.2, 2) {$\alpha_2^\vee$};

\draw [dashed,red] (0,2) -- (0,0);
\draw [dashed] (0,0) -- (0,-2);
\draw [dashed] (1,2) -- (1,-2);
\draw [dashed] (-1,2) -- (-1,-2);

\draw [dashed] (210:2cm) -- (0,0);
\draw [dashed,red] (0,0) -- (30:2cm);

\begin{scope}[shift={(-60:1cm)}]
\draw [dashed] (210:2cm) -- (30:2cm);
\end{scope}

\begin{scope}[shift={(300:1cm)}]
\draw [dashed] (210:2cm) -- (30:2cm);
\end{scope}

\begin{scope}[shift={(120:1cm)}]
\draw [dashed] (210:2cm) -- (30:2cm);
\end{scope}

\draw [dashed] (150:2cm) -- (-30:2cm);

\begin{scope}[shift={(60:1cm)}]
\draw [dashed] (150:2cm) -- (-30:2cm);
\end{scope}

\begin{scope}[shift={(240:1cm)}]
\draw [dashed] (150:2cm) -- (-30:2cm);
\end{scope}
\draw [fill = red] (0,0) -- (30:1.17cm) -- (0,1.17) -- (0,0);

\node[scale=0.7] at (0.4,0.6) {$A_{\mathrm{fund}}$};

\end{tikzpicture}
\end{center}

\bigskip

The dominant chamber $C_{\mathrm{dom}}$ is enclosed by the two red dashed (half) hyperplanes, and each equilateral triangle bounded by dashed lines is an alcove. 
\\

A \textit{gallery} is intuitively a path built up of consecutive alcoves sharing a common facet in $\mathscr{A}^{\mathrm{aff}}$; we will now try to formalize this notion. Fix $\lambda$ a dominant coweight in $\Lambda$, and let $\lambda_{\mathrm{fund}} = \mathrm{type}(\lambda)$. Then there exists a unique shortest element $w_\lambda \in W^{\mathrm{aff}}$ such that $w_\lambda \cdot \lambda_{\mathrm{fund}} = \lambda$.

If we choose a reduced decomposition
$$w_\lambda = s_{i_p} \cdots s_{i_1}$$
we would equivalently be choosing a (minimal) sequence of alcoves and galleries:
$$\gamma_\lambda = \{0\} = F_0 \subset \overline{\Gamma_0} \supset F_1 \subset \overline{\Gamma_1} \supset F_2 \subset \cdots \supset F_p \subset \overline{\Gamma_p} \supset F_{p+1} = \{\lambda\}$$
where $\Gamma_0 = A_{\mathrm{fund}}$, $\Gamma_j = s_{i_j} \cdots s_{i_1} \Gamma_0$ are alcoves, and $F_j = H_{s_{i_j}} \cap \overline{\Gamma_{j-1}}$ are the faces with respect to which the reflections $s_{i_j}$ are carried out. We call a sequence $\delta$ of alcoves and faces starting at 0 such that each pair of consecutive alcoves share a common face a \textit{gallery}. If $\delta = \gamma_\lambda$ for some dominant coweight $\lambda$ and some choice of decomposing $w_\lambda$, we call it a \textit{based gallery}. It is not hard to see that based galleries lie entirely in the dominant chamber $C_{\mathrm{dom}}$. Note that given a dominant coweight $\lambda$, we have two sets of equivalent datum:
$$\left\{\begin{matrix}\text{Based galleries}\\ \text{ending at } \lambda \end{matrix}\right\} \leftrightarrow \left\{\begin{matrix}\text{Reduced expressions of}\\ \text{the element } w_\lambda \in W^{\mathrm{aff}} \end{matrix}\right\}$$

\subsection{LS Galleries} \label{subsect: LSgalleries} Let $\mu$ be a coweight, and let
$$\delta = \{0\} = G_0 \subset \overline{\Delta_0} \supset G_1 \subset \overline{\Delta_1} \supset G_2 \subset \cdots \supset G_p \subset \overline{\Delta_p} \supset G_{p+1} = \{\mu\}$$
be an arbitrary gallery. We define the \textit{weight} and \textit{type} of $\delta$ as follows:
$$\mathrm{wt}(\delta) = \mu \quad \mathrm{type}(\delta) = (\mathrm{type}(G_0), \mathrm{type}(\Delta_0), \mathrm{type}(G_1), \ldots, \mathrm{type}(G_{p+1}))$$

Fix another dominant coweight $\lambda$, and fix
$$\gamma_\lambda = \{0\} = F_0 \subset \overline{\Gamma_0} \supset F_1 \subset \overline{\Gamma_1} \supset F_2 \subset \cdots \supset F_p \subset \overline{\Gamma_p} \supset F_{p+1} = \{\lambda\}$$
a based gallery; recall that this is equivalent to fixing a reduced decomposition of $w_\lambda \in W^{\mathrm{aff}}$. Define
$$\Gamma(\gamma_\lambda) \defeq \{\delta \text{ a gallery }: \mathrm{type}(\delta) = \mathrm{type}(\gamma_\lambda)\}$$

We say that $\delta$ is \textit{positively folded} if for every $j = 1, \ldots, p$, whenever $\Delta_j = \Delta_{j-1}$, there exists some (positive) affine root $\beta \in \Phi_+ \times \ZZ$ such that $G_j \subset H_\beta \text{  but  } \Delta_j \not \subset H_\beta^-$. At each step $j$, we let $\Phi^{\mathrm{aff}}_+(j) \defeq \{\beta \in \Phi_+ \times \ZZ: G_j \subset H_\beta \text{  but  } \Delta_j \not \subset H_\beta^-\}$ and define
$$\mathrm{dim} \, \delta \defeq \sum_{j=0}^p |\Phi^{\mathrm{aff}}_+(j)|$$
Now suppose $\delta$ is a positively folded gallery in $\Gamma(\gamma_\lambda)$. We say $\delta$ is an \textit{LS gallery} if 
$$\mathrm{dim} \, \gamma_\lambda - \mathrm{dim} \, \delta = \mathrm{ht}(\lambda-\mathrm{wt}(\delta))$$
Note that for based galleries $\gamma_\lambda$, we have $\mathrm{dim} \, \gamma_\lambda = |\Phi^+| + \ell(w_\lambda)$. We denote by $\Gamma_{\mathrm{LS}}(\gamma_\lambda)$ the set of LS galleries of type $\gamma_\lambda$.
\\

The following theorem sums up the theory of LS galleries: 

\begin{thm}[Gaussent, Littelmann] \label{thm: LSgalleries} There is a bijection 
$$Z: \Gamma_{LS}(\gamma_\lambda) \overset{\sim}{\longrightarrow} \bigcup_{\nu \in \Lambda} \mathrm{Irr}(\overline{S_\nu \cap \mathrm{Gr}_{\leq \lambda}})$$
such that those LS galleries with weight $\nu$ correspond to irreducible components in $\overline{S_\nu \cap \mathrm{Gr}_{\leq \lambda}}$.
\end{thm}

\noindent
\textbf{Example.} We've given some convoluted conditions to arrive at the definition of LS galleries, so let's see some examples and nonexamples. Consider the case of type $A_2$, and $G = \mathrm{SL}_3$; we will be using the picture from the example in the previous section. Consider the following based gallery $\gamma_\lambda$:\\

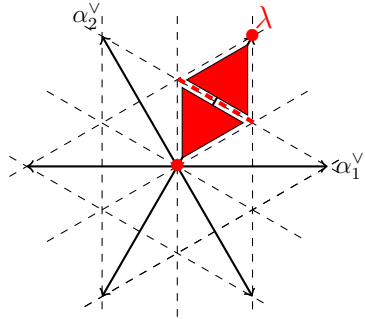
\begin{center}
\begin{tikzpicture}[scale=1]
\foreach\ang in {60,120,...,360}{
     \draw [thick] [->] (0,0) -- (\ang:2cm);
    }
\node[scale=0.8] at (2.3,0) {$\alpha_1^\vee$};
\node[scale=0.8] at (-1.2, 2) {$\alpha_2^\vee$};
\draw [dashed] (0,2) -- (0,0);
\draw [dashed] (0,0) -- (0,-2);
\draw [dashed] (1,2) -- (1,-2);
\draw [dashed] (-1,2) -- (-1,-2);

\draw [dashed] (210:2cm) -- (0,0);
\draw [dashed] (0,0) -- (30:2cm);

\begin{scope}[shift={(-60:1cm)}]
\draw [dashed] (210:2cm) -- (30:2cm);
\end{scope}

\begin{scope}[shift={(300:1cm)}]
\draw [dashed] (210:2cm) -- (30:2cm);
\end{scope}

\begin{scope}[shift={(120:1cm)}]
\draw [dashed] (210:2cm) -- (30:2cm);
\end{scope}

\draw [dashed] (150:2cm) -- (-30:2cm);

\begin{scope}[shift={(60:1cm)}]
\draw [dashed] (150:2cm) -- (-30:2cm);
\end{scope}

\begin{scope}[shift={(240:1cm)}]
\draw [dashed] (150:2cm) -- (-30:2cm);
\end{scope}

\begin{scope}[shift={(60:0.13cm)}]
\draw [fill = red, scale = 0.8] (0,0) -- (30:1.17cm) -- (0,1.17) -- (0,0);
\end{scope}

\begin{scope}[shift={(60:2cm)}]
\begin{scope}[shift={(60:-0.13cm)}]
\draw [fill = red, scale = 0.8] (0,0) -- (210:1.17cm) -- (0,-1.17) -- (0,0);
\end{scope}
\end{scope}

\node [red] at (0,0) {$\bullet$};
\node [red] at (60:2cm) {$\bullet$};
\node [red] at (60:2.3cm) {$\lambda$};
\draw [dashed, ultra thick, red] (30:1.17cm) -- (0,1.17);
\end{tikzpicture}
\end{center}

Then we can easily write down five more galleries in $\Gamma(\gamma_\lambda)$ by applying reflections: 

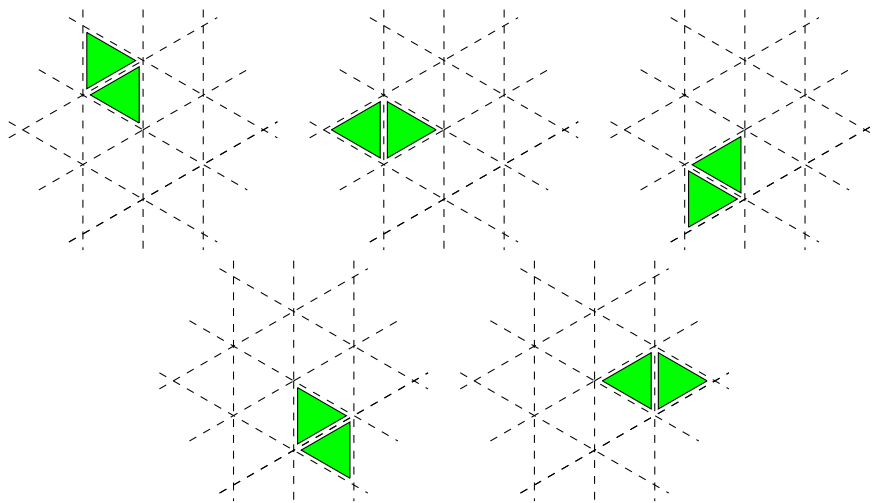
\begin{center}
\begin{tikzpicture}[scale=0.8]
\draw [dashed] (0,2) -- (0,0);
\draw [dashed] (0,0) -- (0,-2);
\draw [dashed] (1,2) -- (1,-2);
\draw [dashed] (-1,2) -- (-1,-2);

\draw [dashed] (210:2cm) -- (0,0);
\draw [dashed] (0,0) -- (30:2cm);

\begin{scope}[shift={(-60:1cm)}]
\draw [dashed] (210:2cm) -- (30:2cm);
\end{scope}

\begin{scope}[shift={(300:1cm)}]
\draw [dashed] (210:2cm) -- (30:2cm);
\end{scope}

\begin{scope}[shift={(120:1cm)}]
\draw [dashed] (210:2cm) -- (30:2cm);
\end{scope}

\draw [dashed] (150:2cm) -- (-30:2cm);

\begin{scope}[shift={(60:1cm)}]
\draw [dashed] (150:2cm) -- (-30:2cm);
\end{scope}

\begin{scope}[shift={(240:1cm)}]
\draw [dashed] (150:2cm) -- (-30:2cm);
\end{scope}

\begin{scope}[rotate=60]
\begin{scope}[shift={(60:0.13cm)}]
\draw [fill = green, scale = 0.8] (0,0) -- (30:1.17cm) -- (0,1.17) -- (0,0);
\end{scope}

\begin{scope}[shift={(60:2cm)}]
\begin{scope}[shift={(60:-0.13cm)}]
\draw [fill = green, scale = 0.8] (0,0) -- (210:1.17cm) -- (0,-1.17) -- (0,0);
\end{scope}
\end{scope}
\end{scope}

\begin{scope}[shift={(5,0)}]
\draw [dashed] (0,2) -- (0,0);
\draw [dashed] (0,0) -- (0,-2);
\draw [dashed] (1,2) -- (1,-2);
\draw [dashed] (-1,2) -- (-1,-2);

\draw [dashed] (210:2cm) -- (0,0);
\draw [dashed] (0,0) -- (30:2cm);

\begin{scope}[shift={(-60:1cm)}]
\draw [dashed] (210:2cm) -- (30:2cm);
\end{scope}

\begin{scope}[shift={(300:1cm)}]
\draw [dashed] (210:2cm) -- (30:2cm);
\end{scope}

\begin{scope}[shift={(120:1cm)}]
\draw [dashed] (210:2cm) -- (30:2cm);
\end{scope}

\draw [dashed] (150:2cm) -- (-30:2cm);

\begin{scope}[shift={(60:1cm)}]
\draw [dashed] (150:2cm) -- (-30:2cm);
\end{scope}

\begin{scope}[shift={(240:1cm)}]
\draw [dashed] (150:2cm) -- (-30:2cm);
\end{scope}

\begin{scope}[rotate=120]
\begin{scope}[shift={(60:0.13cm)}]
\draw [fill = green, scale = 0.8] (0,0) -- (30:1.17cm) -- (0,1.17) -- (0,0);
\end{scope}

\begin{scope}[shift={(60:2cm)}]
\begin{scope}[shift={(60:-0.13cm)}]
\draw [fill = green, scale = 0.8] (0,0) -- (210:1.17cm) -- (0,-1.17) -- (0,0);
\end{scope}
\end{scope}
\end{scope}
\end{scope}

\begin{scope}[shift={(10,0)}]
\draw [dashed] (0,2) -- (0,0);
\draw [dashed] (0,0) -- (0,-2);
\draw [dashed] (1,2) -- (1,-2);
\draw [dashed] (-1,2) -- (-1,-2);

\draw [dashed] (210:2cm) -- (0,0);
\draw [dashed] (0,0) -- (30:2cm);

\begin{scope}[shift={(-60:1cm)}]
\draw [dashed] (210:2cm) -- (30:2cm);
\end{scope}

\begin{scope}[shift={(300:1cm)}]
\draw [dashed] (210:2cm) -- (30:2cm);
\end{scope}

\begin{scope}[shift={(120:1cm)}]
\draw [dashed] (210:2cm) -- (30:2cm);
\end{scope}

\draw [dashed] (150:2cm) -- (-30:2cm);

\begin{scope}[shift={(60:1cm)}]
\draw [dashed] (150:2cm) -- (-30:2cm);
\end{scope}

\begin{scope}[shift={(240:1cm)}]
\draw [dashed] (150:2cm) -- (-30:2cm);
\end{scope}

\begin{scope}[rotate=180]
\begin{scope}[shift={(60:0.13cm)}]
\draw [fill = green, scale = 0.8] (0,0) -- (30:1.17cm) -- (0,1.17) -- (0,0);
\end{scope}

\begin{scope}[shift={(60:2cm)}]
\begin{scope}[shift={(60:-0.13cm)}]
\draw [fill = green, scale = 0.8] (0,0) -- (210:1.17cm) -- (0,-1.17) -- (0,0);
\end{scope}
\end{scope}
\end{scope}
\end{scope}
\end{tikzpicture}
\end{center}

\begin{center}
\begin{tikzpicture}[scale=0.8]
\draw [dashed] (0,2) -- (0,0);
\draw [dashed] (0,0) -- (0,-2);
\draw [dashed] (1,2) -- (1,-2);
\draw [dashed] (-1,2) -- (-1,-2);

\draw [dashed] (210:2cm) -- (0,0);
\draw [dashed] (0,0) -- (30:2cm);

\begin{scope}[shift={(-60:1cm)}]
\draw [dashed] (210:2cm) -- (30:2cm);
\end{scope}

\begin{scope}[shift={(300:1cm)}]
\draw [dashed] (210:2cm) -- (30:2cm);
\end{scope}

\begin{scope}[shift={(120:1cm)}]
\draw [dashed] (210:2cm) -- (30:2cm);
\end{scope}

\draw [dashed] (150:2cm) -- (-30:2cm);

\begin{scope}[shift={(60:1cm)}]
\draw [dashed] (150:2cm) -- (-30:2cm);
\end{scope}

\begin{scope}[shift={(240:1cm)}]
\draw [dashed] (150:2cm) -- (-30:2cm);
\end{scope}

\begin{scope}[rotate=240]
\begin{scope}[shift={(60:0.13cm)}]
\draw [fill = green, scale = 0.8] (0,0) -- (30:1.17cm) -- (0,1.17) -- (0,0);
\end{scope}

\begin{scope}[shift={(60:2cm)}]
\begin{scope}[shift={(60:-0.13cm)}]
\draw [fill = green, scale = 0.8] (0,0) -- (210:1.17cm) -- (0,-1.17) -- (0,0);
\end{scope}
\end{scope}
\end{scope}

\begin{scope}[shift={(5,0)}]
\draw [dashed] (0,2) -- (0,0);
\draw [dashed] (0,0) -- (0,-2);
\draw [dashed] (1,2) -- (1,-2);
\draw [dashed] (-1,2) -- (-1,-2);

\draw [dashed] (210:2cm) -- (0,0);
\draw [dashed] (0,0) -- (30:2cm);

\begin{scope}[shift={(-60:1cm)}]
\draw [dashed] (210:2cm) -- (30:2cm);
\end{scope}

\begin{scope}[shift={(300:1cm)}]
\draw [dashed] (210:2cm) -- (30:2cm);
\end{scope}

\begin{scope}[shift={(120:1cm)}]
\draw [dashed] (210:2cm) -- (30:2cm);
\end{scope}

\draw [dashed] (150:2cm) -- (-30:2cm);

\begin{scope}[shift={(60:1cm)}]
\draw [dashed] (150:2cm) -- (-30:2cm);
\end{scope}

\begin{scope}[shift={(240:1cm)}]
\draw [dashed] (150:2cm) -- (-30:2cm);
\end{scope}

\begin{scope}[rotate=300]
\begin{scope}[shift={(60:0.13cm)}]
\draw [fill = green, scale = 0.8] (0,0) -- (30:1.17cm) -- (0,1.17) -- (0,0);
\end{scope}

\begin{scope}[shift={(60:2cm)}]
\begin{scope}[shift={(60:-0.13cm)}]
\draw [fill = green, scale = 0.8] (0,0) -- (210:1.17cm) -- (0,-1.17) -- (0,0);
\end{scope}
\end{scope}
\end{scope}
\end{scope}
\end{tikzpicture}
\end{center}
It's not hard to check that these six galleries, each corresponding to a coroot $s_\beta \cdot \lambda$ for some $\beta \in \Phi$, are in fact LS galleries. Thus, by Theorem \ref{thm: LSgalleries}, we can conclude that each of the galleries presented here gives an irreducible component of $\overline{S_{s_\beta \cdot \lambda} \cap \left(\mathrm{Gr}_{\mathrm{PGL}_3}\right)_{\leq \lambda}}$. In fact we can deduce that for each $\beta$, $\overline{S_{s_\beta \cdot \lambda} \cap \left(\mathrm{Gr}_{\mathrm{PGL}_3}\right)_{\leq \lambda}}$ is irreducible. This is because there are no other LS galleries with weight $s_\beta \cdot \lambda$.

From the dual perspective given by geometric Satake, i.e., the representation theory of $\mathrm{SL}_3$, this is not surprising: we are considering the highest weight representation with weight $\lambda$, which is the adjoint representation on $\mathfrak{sl}_3$. For each root $s_\beta \cdot \lambda$, we have a one dimensional weight space. 

Let us now consider the galleries in $\Gamma(\gamma_\lambda)$ with weight 0; i.e., galleries that start and end at 0. There are six such galleries:

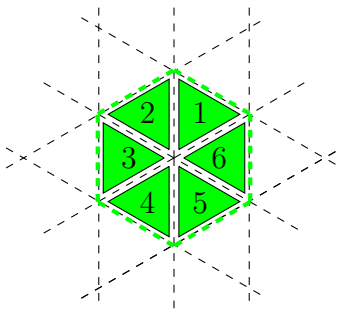
\begin{center}
\begin{tikzpicture}[scale=1]
\draw [dashed] (0,2) -- (0,0);
\draw [dashed] (0,0) -- (0,-2);
\draw [dashed] (1,2) -- (1,-2);
\draw [dashed] (-1,2) -- (-1,-2);

\draw [dashed] (210:2cm) -- (0,0);
\draw [dashed] (0,0) -- (30:2cm);

\begin{scope}[shift={(-60:1cm)}]
\draw [dashed] (210:2cm) -- (30:2cm);
\end{scope}

\begin{scope}[shift={(300:1cm)}]
\draw [dashed] (210:2cm) -- (30:2cm);
\end{scope}

\begin{scope}[shift={(120:1cm)}]
\draw [dashed] (210:2cm) -- (30:2cm);
\end{scope}

\draw [dashed] (150:2cm) -- (-30:2cm);

\begin{scope}[shift={(60:1cm)}]
\draw [dashed] (150:2cm) -- (-30:2cm);
\end{scope}

\begin{scope}[shift={(240:1cm)}]
\draw [dashed] (150:2cm) -- (-30:2cm);
\end{scope}

\foreach\ang in {60,120,...,360}{
\begin{scope}[rotate=\ang]
\draw [dashed, ultra thick, green] (90:1.17cm) -- (30:1.17cm);
\begin{scope}[shift={(60:0.13cm)}]
\draw [fill = green, scale = 0.8] (0,0) -- (30:1.17cm) -- (0,1.17) -- (0,0);
\end{scope}
\end{scope}
}

\node at (0.35,0.6) {1}; 
\node at (-0.35, 0.6) {2};
\node at (-0.6,0) {3};
\node at (-0.35, -0.6) {4};
\node at (0.35, -0.6) {5};
\node at (0.6, 0) {6};

\end{tikzpicture}
\end{center}

If we think about the zero weight space of the adjoint representation, or equivalently the irreducible components of $\overline{S_0 \cap \left(\mathrm{Gr}_{\mathrm{PGL}_3}\right)_{\leq \lambda}}$, we know that there should only be two LS galleries in $\Gamma(\gamma_\lambda)$ of weight 0. 

Indeed, we can quickly eliminate galleries 1, 2, and 6; these three galleries are not positively folded. We can also eliminate gallery 4 as it is not an LS gallery: let $d$ be the dimension of gallery 4, then we have
$$\mathrm{dim} \, \gamma_\lambda - d = 4-4 = 0$$
But $\mathrm{ht}(\lambda-0) = 2$.

\subsection{The Bott-Samelson Variety} \label{subsect: galleriesBSVariety}
In this section, we sketch Gaussent and Littelmann's proof of Theorem \ref{thm: LSgalleries} in \cite{Gaussent-Littelmann}. We alert the reader that this section is optional since it does not add combinatorial content to MV cycles, which is our main goal. In fact, in the previous section we have completely ignored the geometric origins of LS galleries, with the purpose of distilling their combinatorial structures; the disadvantage of this approach is that the purported bijection in Theorem \ref{thm: LSgalleries} is rather opaque. It is crucial in Gaussent and Littelmann's strategy to view LS galleries in the context of Bott-Samelson resolutions of $\mathrm{Gr}_{\leq \lambda}$, which in turn can be viewed as galleries in the affine Bruhat-Tits building $\mathscr{I}^{\mathrm{aff}}$. We assume for this section the theory of affine buildings, and outline the proof in which LS galleries naturally arise, in order to describe the bijection $Z$.

Let's first fix some notation. Let $I$ be the Iwahori subgroup corresponding to our chosen Borel, i.e., $I$ is the preimage of Borel under the speciaization map
\begin{center}
\begin{tikzpicture}
\node at (0,0) {$G(\mathcal{O}) \longrightarrow G(\CC)$};
\node at (0,-0.5) {$t \mapsto 0$};
\end{tikzpicture}
\end{center}

We also fix a choice of lifting of the simple reflections generating $W^{\mathrm{aff}}$, and for every simple reflection $s_i$, we denote by $\tilde{P_i}$ the parahoric subgroup generated by the Iwahori and $s_i$ inside $G(F)$.

Let $\gamma_\lambda$ be a based gallery ending at $\lambda$; this gives us a choice of a reduced expression $w_\lambda = s_{i_p} \cdots s_{i_1}$. 
The \textit{Bott-Samelson variety} associated to $\gamma_\lambda$ is the following smooth projective variety
$$\hat{\Sigma}(\gamma_\lambda) = G(\mathcal{O}) \times_{I} \tilde{P}_{i_1} \times_{I} \cdots \times_{I} \tilde{P}_{i_p}/I$$
Note that we have a $G(\mathcal{O})$-equivariant map 
$$\pi: \hat{\Sigma}(\gamma_\lambda) \longrightarrow \mathrm{Gr}_{\leq \lambda}$$
$$[g_0, g_1, \ldots, g_p] \mapsto g_0\cdots g_p \cdot t^{\lambda_{\mathrm{fund}}}$$

Recall that the affine Bruhat-Tits building $\mathscr{I}^{\mathrm{aff}}$ is a certain quotient of $G(F) \times \mathscr{A}^{\mathrm{aff}}$, with $\left(\mathscr{I}^{\mathrm{aff}}\right)^T = \mathscr{A}^{\mathrm{aff}}$ (see Definition 4 in \cite{Gaussent-Littelmann} for the precise description of this quotient). We can then view $\hat{\Sigma}(\gamma_\lambda)$ as galleries in $\mathscr{I}^{\mathrm{aff}}$ by sending an element $[g_0, \ldots, g_p]$ of the Bott-Samelson variety to the gallery
$$\{0\}, \, g_0 \cdot A_{\mathrm{fund}}, \, g_0g_1 \cdot A_{\mathrm{fund}}, \ldots, \{g_0\cdots g_p \lambda_{\mathrm{fund}}\}$$

As with any Bruhat-Tits building, we have retraction maps onto a given apartment; let $r: \mathscr{I}^{\mathrm{aff}} \to \mathscr{A}^{\mathrm{aff}}$ be the retraction onto $\mathscr{A}^{\mathrm{aff}} = \left(\mathscr{I}^{\mathrm{aff}}\right)^T$. This extends to a retraction map (Section 7 in \cite{Gaussent-Littelmann})
$$r: \hat{\Sigma}(\gamma_\lambda) \to \hat{\Sigma}(\gamma_\lambda)^T$$
Using the language developed here, Theorem \ref{thm: LSgalleries} can be rephrased more accurately as follows:
\begin{thm} [Gaussent, Littelmann]
The map 
\begin{center}
\begin{tikzpicture}
\node at (0,0) {$Z: \Gamma_{LS}(\gamma_\lambda) \to \bigcup_{\nu} \mathrm{Irr}(\overline{S_\nu \cap \mathrm{Gr}_{\leq \lambda}})$};
\node at (0,-0.7) {$\delta \mapsto \overline{\pi(r^{-1}(\delta))}$};
\end{tikzpicture}
\end{center}
is a bijection of sets. Furthermore, if $\mathrm{wt}(\delta) = \nu$, then $Z(\delta)$ is in $\mathrm{Irr}(\overline{S_\nu \cap \mathrm{Gr}_{\leq \lambda}})$.
\end{thm}

\section{Crystals}
So far, we've introduced several structures associated to MV cycles on $\gr$: weight spaces of $\hat{G}$ representations, MV polytopes, modules over $\Pi Q$, and LS galleries. These objects in fact inherit crystal structures as defined by Kashiwara, and are isomorphic as crystals via the various correspondences we've discussed. In this section, we review the definition of crystals following \cite{Kashiwara}, define the crystal structures on MV polytopes and LS galleries, and explain several isomorphisms of crystals. Readers who are experienced with the basic theory of crystals may wish to proceed to Section \ref{subsect: crystalMV} directly.  

\subsection{The Crystal $B(-\infty)$} \label{subsect: infty)}
The formal definition of crystals is quite complicated, but is very concretely modeled after the theory of highest weight representations of semisimple Lie algebras. Thus, we first give some motivation before defining an abstract crystal.

Let us fix notations for this section. Let $P$ be a free $\ZZ$-module (the ``weight lattice"), $I$ an index set, a set of ``simple roots" $\alpha_i \in P$ for each $i \in I$, a set of ``simple coroots" $h_i \in P^*$ for each $i \in I$, and $(\cdot, \cdot): P \times P \to \QQ$ a bilinear symmetric form. We also write $\langle \cdot, \cdot \rangle$ for the canonical pairing $P \times P^* \to \QQ$. We assume that the above datum satisfy the following conditions:
$$(\alpha_i, \alpha_i) \in 2\NN \text{ for each } i.$$ 
$$\langle h_i, \lambda \rangle = \frac{2(\alpha_i, \lambda)}{(\alpha_i, \alpha_i)} \text{ for each } i \in I \text{ and } \lambda \in P.$$
$$(\alpha_i, \alpha_j) \leq 0 \text{ for } i, j \in I \text{ with } i \neq j.$$

Let $\g$ be a semisimple Lie algebra, and $\mathfrak{h}$ a Cartan subalgebra. For now we can pretend that $P$ is a genuine weight lattice, the $\alpha_i$'s are simple roots, and the $h_i$'s are simple coroots in a Lie-theoretic context. Suppose we are given some finite dimensional irreducible highest weight representation $V$; we can then decompose $V$ as follows: 
$$V = \bigoplus_{\lambda \in \mathfrak{h}^*} V_\lambda, \text{ where } V_\lambda = \{v \in V: h\cdot v = \lambda(h)v \text{ for all } h \in \mathfrak{h}\}$$ 
Let $B \subseteq \mathfrak{h}^*$ be the weights of $V$, i.e., the finite subset consisting of those $\lambda$ for which $V_\lambda$ is nonzero. Choosing Chevalley generators $e_i, f_i, h_i$, we can check by standard calculations that for every $\lambda \in B$, 
$$e_i \cdot V_\lambda  \subseteq V_{\lambda+\alpha_i} \\ \text{  and  } \\ f_i \cdot V_{\mu} \subseteq V_{\mu - \alpha_i}$$
Fixing an $i$, we can then visualize the weight spaces as a ``string":

\begin{center}
\begin{tikzpicture}
\node at (0,0) (a) {$\bullet$};
\node at (90:0.5cm) {$\lambda$};
\node at (1.5,0) (b) {$\bullet$};
\node at (1.5,0.5) {$\lambda+\alpha_i$};
\node at (-1.5,0) (c) {$\bullet$};
\node at (-1.5,0.5) {$\lambda-\alpha_i$};
\draw [thick] (a)--(b);
\draw [thick] (a)--(c);
\draw [thick] (b) -- (3,0);
\draw [thick] (c) -- (-3,0);
\node at (3.5,0) {$\cdots$};
\node at (-3.5,0) {$\cdots$};
\end{tikzpicture}
\end{center}
where applications of $e_i$ move you to the right, and applications of $f_i$ move you to the left. Since $V$ is assumed finite dimensional (or alternatively by Serre's relations), for each weight $\lambda$ we can define $\varepsilon_i(\lambda)$ and $\varphi_i(\lambda)$ as the largest natural numbers such that
$$(e_i)^{\varepsilon_i(\lambda)} \cdot V_\lambda \neq 0 \\ \text{  and  } \\ (f_i)^{\varphi_i(\lambda)} \cdot V_\lambda \neq 0$$
The definition of a crystal formalizes the intuitive notions discussed so far.

\begin{defn} A crystal is a set $B$ along with the following datum:
\begin{enumerate}
\item $\mathrm{wt}: B \to P$
\item $\varepsilon_i: B \to \ZZ \sqcup \{-\infty\}$
\item $\varphi_i: B \to \ZZ \sqcup \{-\infty\}$
\item $\tilde{e}_i: B \to B \sqcup \{0\}$
\item $\tilde{f}_i: B \to B \sqcup \{0\}$
\end{enumerate}
satisfying the following axioms:
\begin{enumerate}
\item $\varphi_i(b) = \varepsilon_i(b) + \langle h_i, \mathrm{wt}(b) \rangle \text{  for each  } i$
\item If $b \in B$ satisfies $\tilde{e}_ib \neq 0$, then:
$$\varepsilon_i(\tilde{e}_ib) = \varepsilon_i(b)-1, \quad \varphi_i(\tilde{e}_ib) = \varphi_i(b)+1, \quad \mathrm{wt}(\tilde{e}_ib) = \mathrm{wt}(b) + \alpha_i.$$
\item If $b \in B$ satisfies $\tilde{f}_ib \neq 0$, then:
$$\varepsilon_i(\tilde{f}_ib) = \varepsilon_i(b)+1, \quad \varphi_i(\tilde{f}_ib) = \varphi_i(b)-1, \quad \mathrm{wt}(\tilde{f}_ib) = \mathrm{wt}(b) - \alpha_i.$$
\item For $b_1, b_2 \in B$, $b_2 = \tilde{f}_ib_1$ if and only if $b_1 = \tilde{e}_ib_2$.
\item If $\varphi_i(b) = -\infty$, then $\tilde{e}_ib = \tilde{f}_i(b) = 0$.
\end{enumerate}
\end{defn}

Let $B_1, B_2$ be two crystals. A morphism $\psi: B_1 \to B_2$ is a set-theoretic map $B_1 \sqcup \{0\} \to B_2 \sqcup \{0\}$ satisfying the following conditions:
\begin{enumerate}
\item $\psi(0) = 0$.
\item For all $b \in B_1$, we have
$$\mathrm{wt}(\psi(b)) = \mathrm{wt}(b), \quad \varepsilon_i(\psi(b)) = \varepsilon_i(b), \quad \varphi_i(\psi(b)) = \varphi_i(b).$$
\item Let $b, b' \in B_1$ be nonzero and $i \in I$ be such that $\tilde{f}_i(b) = b'$. If $\psi(b), \psi(b')$ are nonzero, then
$$\tilde{f}_i(\psi(b)) = \psi(b').$$
\end{enumerate}
If $\psi$ commutes with all actions of $\tilde{e}_i, \tilde{f}_i$, then we say that $\psi$ is \textit{strict}. If $\psi$ is strict and bijective on the underlying set, we say that $\psi$ is an isomorphism. \\

Let's first give three easy examples of crystals.

\noindent
\textbf{Example.}
Let's start with the most trivial example. Let $\lambda$ be an arbitrary weight, and let $T(\lambda) = \{t_\lambda\}$ be the crystal defined by
$$\mathrm{wt}(t_\lambda) = \lambda, \quad \tilde{e}_i \cdot t_\lambda = \tilde{f}_i \cdot t_\lambda = 0$$
$$\varphi(t_\lambda) = \varepsilon(t_\lambda) = -\infty$$
It is trivial to see that this defines a crystal. 

For a second example, let $i$ be an element of the index set $I$. Define the crystal $B_i = \{b_i(n): n \in \ZZ\}$ as follows:
$$\mathrm{wt}(b_i(n)) = n, \quad \tilde{e}_i \cdot b_i(n) = b_i(n+1), \quad \tilde{f}_i b_i(n) = b_i(n-1)$$
$$\tilde{e}_j \cdot b_i(n) = \tilde{f}_j \cdot b_i(n) = 0, \text{ for } i \neq j$$
$$\varphi_i(b_i(n)) = n, \quad \varepsilon_i(b_i(n)) = -n; \quad \varphi_j(b_i(n)) = \varepsilon_j(b_i(n)) = -\infty$$
In other words, we single out the $i$th ``$\mathfrak{sl}_2$-string".

Another natural example comes from the motivating example we sketched before the definition of a crystal. Let $\lambda$ be a dominant integral weight for some semisimple Lie algebra $\g$. Let $B(\lambda)$ be the set of weights in the finite dimensional simple module with highest weight $\lambda$. The root operators $\tilde{e}_i, \tilde{f}_i$ act on the weights as they do in the representation, and $\varphi_i, \varepsilon_i$ record the largest integer such that a particular weight is not annihilated by the respective root operators. Note that every element in the crystal is generated by the highest weight element $b_\lambda$. \\

We now introduce the crystal $B(-\infty)$ associated to \textit{crystal bases} for the quantum group $U_q(\mathfrak{n}^+)$ as defined by Kashiwara and Saito in \cite{Kashiwara-Saito}.\\

\noindent
\textbf{Example.} The crystal $B(-\infty)$ is a \textit{lowest weight crystal}. This means that it has a ``lowest weight element", denoted $u_\infty$, with weight $0$, such that 
\begin{enumerate}
\item $\tilde{f}_i \cdot u_\infty = 0$ for all $i$.
\item All other elements can be obtained from $u_\infty$ by applying $\tilde{e}_i$'s.
\end{enumerate}
Instead of defining $B(-\infty)$, we describe the essential properties that uniquely characterize it (this is given by Proposition 3.2.3. in \cite{Kashiwara-Saito}):
\begin{enumerate}
\item The map $\Psi_i: B(-\infty) \to B(-\infty) \otimes B_i$ given by $u_\infty \mapsto u_\infty \otimes b_i(0)$ is a strict embedding, with image contained in $B \times \{b_i(n): n \in \ZZ_{\geq 0}\}$.
\item For any $b \neq u_\infty$ in $B$, there is an $i$ such that $\Psi_i(b) = b' \otimes \tilde{e}_i^nb_i(0)$ for some $n > 0$.
\end{enumerate}
Analogously, one can define a \textit{highest weight crystal} $B(\infty)$ by dualizing all the statements above (switch $\tilde{f}_i$'s for $\tilde{e}_i$'s, and vice versa). $B(\infty)$ is the universal highest weight crystal in the following sense: let $\lambda$ be a dominant integral weight. Then there is an embedding
$$B(\lambda) \longrightarrow B(\infty) \otimes T(\lambda)$$
$$b_\lambda \mapsto u_\infty \otimes t_\lambda$$

\subsection{Several Crystals} \label{subsect: crystalMV}
Now that we've defined and seen some examples of crystals, we can finally state one of the central results of this note. The result is due to various combinations of the authors referenced so far: Anderson, Baumann, Gaussent, Kamnitzer, Kashiwara, Lusztig, Saito and many others. 

\begin{thm} \label{thm: main}
Let $G$ be a complex reductive algebraic group, and $Q$ a quiver whose underlying graph is the Dynkin diagram associated to $\hat{G}$, the Langlands dual of $G$. Then the set of MV cycles on $\gr$, the set of MV polytopes associated to $\gr$, the set of irreducible components in Lusztig's nilpotent variety $\mathrm{Irr}(\mathfrak{B})$, and the set of LS galleries all admit crystal structures, and are all isomorphic as crystals. In fact, they are isomorphic to the crystal $B(-\infty)$ of crystal bases of $\hat{\g} = \mathrm{Lie}(\hat{G})$. 
\end{thm}

It is worth noting that the crystal $B(-\infty)$ has no nontrivial automorphisms \cite{Kashiwara-Saito}, so these isomorphisms are unique. Unfortunately, the presentation for crystal operations is computationally explicit only for LS galleries; from the perspective of combinatorializing MV cycles, this is rather inconvenient, so we only give the construction for LS galleries here. For example, the crystal operations on an MV polytope are computed by defining it on certain faces, then using structure theorems (``tropical Pl\"ucker relations") of MV polytopes to reconstruct the rest of the polytope (see Section 5 of \cite{Baumann-Kamnitzer-Tingley}). For the crystal structures on MV cycles and $\Pi Q$-modules, we refer the reader to Section 13 of \cite{Braverman-Finkelberg-Gaitsgory} and Section 6 of \cite{Baumann-Kamnitzer-Tingley} respectively. \\



Let's describe the crystal structure on the set of LS galleries $\mathcal{LS}$. Let $\mu$ be a coweight, and let
$$\delta = \{0\} = G_0 \subset \overline{\Delta_0} \supset G_1 \subset \overline{\Delta_1} \supset G_2 \subset \cdots \supset G_p \subset \overline{\Delta_p} \supset G_{p+1} = \{\mu\}$$
be an arbitrary gallery in $\Gamma(\gamma_\lambda)$ for some based gallery $\gamma_\lambda$. Define the crystal operations
$$\mathrm{wt}: \Gamma(\gamma_\lambda) \to P \quad \varepsilon_i, \varphi_i: \Gamma(\gamma_\lambda) \to \ZZ \sqcup \{-\infty\} \quad \tilde{e}_i, \tilde{f}_i: \Gamma(\gamma_\lambda) \to \Gamma(\gamma_\lambda)$$
as follows: fix a simple root $\alpha = \alpha_i$, and let $m \in \ZZ$ be the smallest integer such that the hyperplane $H_{\alpha,m}$ contains a face $G_j$, where $0 \leq j \leq p+1$.

\begin{enumerate}
\item $\mathrm{wt}(\delta) = \mu$ as before.
\item $\varepsilon_i(\delta) = -m, \, \varphi_i(\delta) = \langle \alpha, \mu\rangle -m$.
\item If $m = 0$, $\tilde{e}_i \delta$ is undefined. Otherwise, find $k \in \{1, \ldots, p+1\}$ minimal such that $G_k \subset H_{\alpha, m}$, and find $j \in \{0, \ldots, k-1\}$ maximal such that $G_j \subset H_{\alpha, m+1}$. Then we define $\tilde{e}_i \delta$ as
$$\{0\} = G_0 \subset \overline{\Delta_0} \supset G_1 \cdots \supset G_j \subset s_{\alpha,m+1}(\overline{\Delta_j}) \supset s_{\alpha,m+1}(G_{j+1}) \subset \cdots \supset s_{\alpha,m+1}(G_{k-1})$$ 
$$\subset s_{\alpha, m+1}(\overline{\Delta_{k-1}}) \supset \tau_{\alpha^\vee}(G_k) \subset \tau_{\alpha^\vee}(\overline{\Delta_k}) \supset \cdots \subset \tau_{\alpha^\vee}(\overline{\Delta_p}) \supset \tau_{\alpha^\vee}(G_{p+1}) = \{\mu+\alpha^\vee\}$$
In other words, we reflect all faces between $G_j$ and $G_k$ by $H_{\alpha,m+1}$, and translate all faces after $G_k$ by $\alpha^\vee$.
\item If $m = \langle \alpha, \mu\rangle$, then $\tilde{f}_i\delta$ is not defined. Otherwise find $j \in \{0, \ldots, p\}$ maximal such that $G_j \subset H_{\alpha, m}$, and find $k \in \{j+1, \ldots, p+1\}$ minimal such that $G_k \subset H_{\alpha, m+1}$. Then we define $\tilde{f}_i \delta$ as
$$\{0\} = G_0 \subset \overline{\Delta_0} \supset G_1 \cdots \supset G_j \subset s_{\alpha,m}(\overline{\Delta_j}) \supset s_{\alpha,m}(G_{j+1}) \subset \cdots \supset s_{\alpha,m}(G_{k-1})$$ 
$$\subset s_{\alpha, m}(\overline{\Delta_{k-1}}) \supset \tau_{-\alpha^\vee}(G_k) \subset \tau_{-\alpha^\vee}(\overline{\Delta_k}) \supset \cdots \subset \tau_{-\alpha^\vee}(\overline{\Delta_p}) \supset \tau_{-\alpha^\vee}(G_{p+1}) = \{\mu-\alpha^\vee\}$$
In other words, we reflect all faces between $G_j$ and $G_k$ by $H_{\alpha,m}$, and translate all faces after $G_k$ by $-\alpha^\vee$.
\end{enumerate}

\noindent
\textbf{Example.} These definitions look terrible; let's look at what it's doing on very explicit examples. The following LS galleries of type $\lambda = (1,0,-1)$ in type $A_2$ illustrate a simple case of applying operations in the two directions available $\alpha_1, \alpha_2$:

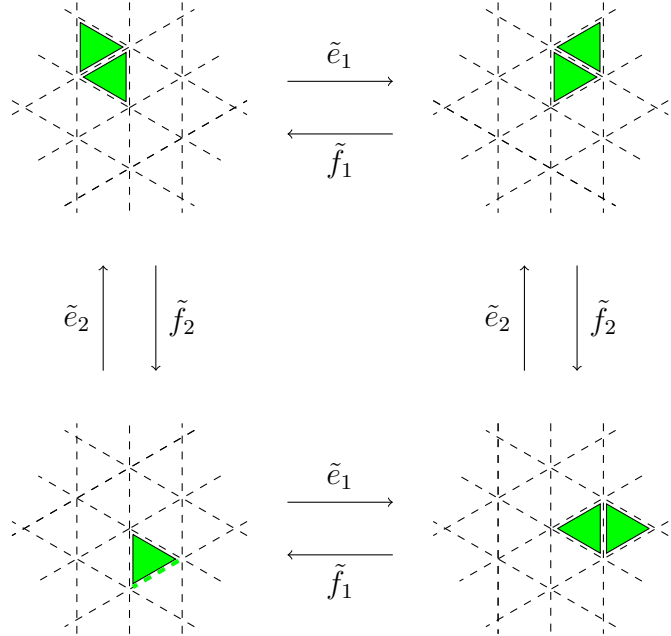
\begin{center}
\begin{tikzpicture}[scale=0.7]
\draw [dashed] (0,2) -- (0,0);
\draw [dashed] (0,0) -- (0,-2);
\draw [dashed] (1,2) -- (1,-2);
\draw [dashed] (-1,2) -- (-1,-2);

\draw [dashed] (210:2cm) -- (0,0);
\draw [dashed] (0,0) -- (30:2cm);

\begin{scope}[shift={(-60:1cm)}]
\draw [dashed] (210:2cm) -- (30:2cm);
\end{scope}

\begin{scope}[shift={(300:1cm)}]
\draw [dashed] (210:2cm) -- (30:2cm);
\end{scope}

\begin{scope}[shift={(120:1cm)}]
\draw [dashed] (210:2cm) -- (30:2cm);
\end{scope}

\draw [dashed] (150:2cm) -- (-30:2cm);

\begin{scope}[shift={(60:1cm)}]
\draw [dashed] (150:2cm) -- (-30:2cm);
\end{scope}

\begin{scope}[shift={(240:1cm)}]
\draw [dashed] (150:2cm) -- (-30:2cm);
\end{scope}

\begin{scope}[rotate=60]
\begin{scope}[shift={(60:0.13cm)}]
\draw [fill = green, scale = 0.8] (0,0) -- (30:1.17cm) -- (0,1.17) -- (0,0);
\end{scope}

\begin{scope}[shift={(60:2cm)}]
\begin{scope}[shift={(60:-0.13cm)}]
\draw [fill = green, scale = 0.8] (0,0) -- (210:1.17cm) -- (0,-1.17) -- (0,0);
\end{scope}
\end{scope}
\end{scope}

\begin{scope}[shift={(0:8cm)}, rotate = -60]
\draw [dashed] (0,2) -- (0,0);
\draw [dashed] (0,0) -- (0,-2);
\draw [dashed] (1,2) -- (1,-2);
\draw [dashed] (-1,2) -- (-1,-2);

\draw [dashed] (210:2cm) -- (0,0);
\draw [dashed] (0,0) -- (30:2cm);

\begin{scope}[shift={(-60:1cm)}]
\draw [dashed] (210:2cm) -- (30:2cm);
\end{scope}

\begin{scope}[shift={(300:1cm)}]
\draw [dashed] (210:2cm) -- (30:2cm);
\end{scope}

\begin{scope}[shift={(120:1cm)}]
\draw [dashed] (210:2cm) -- (30:2cm);
\end{scope}

\draw [dashed] (150:2cm) -- (-30:2cm);

\begin{scope}[shift={(60:1cm)}]
\draw [dashed] (150:2cm) -- (-30:2cm);
\end{scope}

\begin{scope}[shift={(240:1cm)}]
\draw [dashed] (150:2cm) -- (-30:2cm);
\end{scope}

\begin{scope}[rotate=60]
\begin{scope}[shift={(60:0.13cm)}]
\draw [fill = green, scale = 0.8] (0,0) -- (30:1.17cm) -- (0,1.17) -- (0,0);
\end{scope}

\begin{scope}[shift={(60:2cm)}]
\begin{scope}[shift={(60:-0.13cm)}]
\draw [fill = green, scale = 0.8] (0,0) -- (210:1.17cm) -- (0,-1.17) -- (0,0);
\end{scope}
\end{scope}
\end{scope}
\end{scope}

\begin{scope}[shift={(0:3cm)}]
\draw [->] (0,0.5) -- (2,0.5);
\node at (1,1) {$\tilde{e}_1$};
\draw [->] (2,-0.5) -- (0,-0.5);
\node at (1,-1) {$\tilde{f}_1$};
\end{scope}

\begin{scope}[shift={(8,-8)}, rotate=-120]
\draw [dashed] (0,2) -- (0,0);
\draw [dashed] (0,0) -- (0,-2);
\draw [dashed] (1,2) -- (1,-2);
\draw [dashed] (-1,2) -- (-1,-2);

\draw [dashed] (210:2cm) -- (0,0);
\draw [dashed] (0,0) -- (30:2cm);

\begin{scope}[shift={(-60:1cm)}]
\draw [dashed] (210:2cm) -- (30:2cm);
\end{scope}

\begin{scope}[shift={(300:1cm)}]
\draw [dashed] (210:2cm) -- (30:2cm);
\end{scope}

\begin{scope}[shift={(120:1cm)}]
\draw [dashed] (210:2cm) -- (30:2cm);
\end{scope}

\draw [dashed] (150:2cm) -- (-30:2cm);

\begin{scope}[shift={(60:1cm)}]
\draw [dashed] (150:2cm) -- (-30:2cm);
\end{scope}

\begin{scope}[shift={(240:1cm)}]
\draw [dashed] (150:2cm) -- (-30:2cm);
\end{scope}

\begin{scope}[rotate=60]
\begin{scope}[shift={(60:0.13cm)}]
\draw [fill = green, scale = 0.8] (0,0) -- (30:1.17cm) -- (0,1.17) -- (0,0);
\end{scope}

\begin{scope}[shift={(60:2cm)}]
\begin{scope}[shift={(60:-0.13cm)}]
\draw [fill = green, scale = 0.8] (0,0) -- (210:1.17cm) -- (0,-1.17) -- (0,0);
\end{scope}
\end{scope}
\end{scope}
\end{scope}

\draw [->] (8.5,-3) -- (8.5,-5);
\draw [<-] (7.5,-3) -- (7.5,-5);
\node at (9,-4) {$\tilde{f}_2$};
\node at (7,-4) {$\tilde{e}_2$};

\begin{scope}[shift={(0,-8)}, rotate=180]
\draw [dashed] (0,2) -- (0,0);
\draw [dashed] (0,0) -- (0,-2);
\draw [dashed] (1,2) -- (1,-2);
\draw [dashed] (-1,2) -- (-1,-2);

\draw [dashed] (210:2cm) -- (0,0);
\draw [dashed] (0,0) -- (30:2cm);

\begin{scope}[shift={(-60:1cm)}]
\draw [dashed] (210:2cm) -- (30:2cm);
\end{scope}

\begin{scope}[shift={(300:1cm)}]
\draw [dashed] (210:2cm) -- (30:2cm);
\end{scope}

\begin{scope}[shift={(120:1cm)}]
\draw [dashed] (210:2cm) -- (30:2cm);
\end{scope}

\draw [dashed] (150:2cm) -- (-30:2cm);

\begin{scope}[shift={(60:1cm)}]
\draw [dashed] (150:2cm) -- (-30:2cm);
\end{scope}

\begin{scope}[shift={(240:1cm)}]
\draw [dashed] (150:2cm) -- (-30:2cm);
\end{scope}

\begin{scope}[rotate=60]
\begin{scope}[shift={(60:0.13cm)}]
\draw [fill = green, scale = 0.8] (0,0) -- (30:1.17cm) -- (0,1.17) -- (0,0);
\draw [dashed, ultra thick, green] (90:1cm) -- (30:1cm);
\end{scope}
\end{scope}

\end{scope}

\draw [->] (0.5,-3) -- (0.5,-5);
\draw [<-] (-0.5,-3) -- (-0.5,-5);
\node at (1,-4) {$\tilde{f}_2$};
\node at (-1,-4) {$\tilde{e}_2$};

\begin{scope}[shift={(3,-8)}]
\draw [->] (0,0.5) -- (2,0.5);
\node at (1,1) {$\tilde{e}_1$};
\draw [->] (2,-0.5) -- (0,-0.5);
\node at (1,-1) {$\tilde{f}_1$};
\end{scope}

\end{tikzpicture}
\end{center}

In particular, notice that the crystal operations send LS galleries to LS galleries; this is true in general and proven in \cite{Baumann-Gaussent}.

\subsection*{Acknowledgements} There are many people's help without which this project would not be possible. First of all, thanks to Peter May for his consistent effort in organizing the 2017 University of Chicago REU. I would also like to thank Jingren Chi for his tireless mentorship throughout the 8 weeks, and for introducing me to various excellent sources on geoemtric representation theory. Lastly, I would like to thank Joel Kamnitzer for explaining to me many of his extraordinary results. 


\bibliographystyle{abbrv}
\bibliography{MV}

\end{document}